\documentclass[a4paper,11pt]{article}
\title{The harmonic transvector algebra\\ in two vector variables}
\author{Hendrik De Bie\thanks{Hendrik.DeBie@UGent.be}, David Eelbode\thanks{david.eelbode@uantwerpen.be}, Matthias Roels\thanks{matthias.roels@uantwerpen.be}}
\date{}
\usepackage[a4paper]{geometry}

\usepackage{graphicx}
\usepackage{epsfig}
\usepackage{subfigure}
\usepackage{float}
\restylefloat{figure}
\restylefloat{table}
\usepackage[font=small,format=plain,labelfont=bf,up,textfont=it,up]{caption}

\usepackage{enumitem}

\usepackage{amsmath,amsfonts,amsthm,amssymb,mathdots}
\usepackage{mathrsfs}
\usepackage{mathtools}
\usepackage{accents}
\usepackage{dsfont}

\usepackage{tikz}
\usetikzlibrary{arrows,matrix,positioning,shapes,fit,calc}

\usepackage[english]{babel}

\providecommand{\keywords}[1]{\textbf{Keywords:} #1}
\providecommand{\msc}[1]{\textbf{MSC 2010:} #1}

\newcommand{\inner}[1]{\langle #1\rangle} 
\newcommand{\comm}[1]{\lbrack #1 \rbrack} 
\newcommand{\set}[1]{\left\{ #1 \right\} } 
\newcommand{\norm}[1]{\left| #1 \right|} 
\newcommand{\brac}[1]{\left( #1 \right)} 

\newcommand{\mR}{\mathbb{R}} 
\newcommand{\mN}{\mathbb{N}} 
\newcommand{\mC}{\mathbb{C}} 
\newcommand{\mZ}{\mathbb{Z}} 

\newmuskip\pFqskip
\mathchardef\pFcomma=\mathcode`, 

\newcommand*\pFq[5]{%
  \begingroup
  \begingroup\lccode`~=`,
    \lowercase{\endgroup\def~}{\pFcomma\mkern\pFqskip}%
  \mathcode`,=\string"8000
  {}_{#1}F_{#2}\biggl[\genfrac..{0pt}{}{#3}{#4};#5\biggr]%
  \endgroup
}

\newcommand{\D}{\partial} 
\newcommand{\mE}{\mathbb{E}} 

\newcommand{\mcG}{\mathcal{G}}
\newcommand{\mcH}{\mathcal{H}}
\newcommand{\mcI}{\mathcal{I}}

\newcommand{\mcP}{\mathcal{P}}

\newcommand{\mcU}{\mathcal{U}}
\newcommand{\mcW}{\mathcal{W}}

\newcommand{\mS}{\mathbb{S}}  
\newcommand{\mV}{\mathbb{V}}

\newcommand{\g}{\mathfrak{g}}

\newcommand{\p}{\mathfrak{p}}

\newcommand{\gok}{\mathfrak{k}}
\newcommand{\goh}{\mathfrak{h}}

\newcommand{\gl}{\mathfrak{gl}}
\newcommand{\so}{\mathfrak{so}}
\newcommand{\sym}{\mathfrak{sp}}
\newcommand{\spl}{\mathfrak{sl}}

\newcommand{\gs}{\mathfrak{s}}
\newcommand{\gt}{\mathfrak{t}}

\newcommand{\SO}{\operatorname{SO}}

\newcommand{\Aut}{\operatorname{Aut}}
\newcommand{\End}{\operatorname{End}}
\newcommand{\Span}{\operatorname{Span}}

\newcommand{\nuu}{\vert u \vert}
\newcommand{\nx}{\vert x\vert}

\newtheorem{definition}{Definition}[section]
\newtheorem{proposition}{Proposition}[section]
\newtheorem{theorem}{Theorem}[section]
\newtheorem{lemma}{Lemma}[section]
\newtheorem{remark}{Remark}[section]
\newtheorem{corollary}{Corollary}[section]

\begin{document}
\maketitle

\begin{abstract}
\noindent
The decomposition of polynomials of one vector variable into irreducible modules for the orthogonal group is a crucial result in harmonic analysis which makes use of the Howe duality theorem and leads to the study of spherical harmonics. The aim of the present paper is to describe a decomposition of polynomials in two vector variables and to obtain projection operators on each of the irreducible components. To do so, a particular transvector algebra will be used as a new dual partner for the orthogonal group leading to a generalisation of the classical Howe duality. The results are subsequently used to obtain explicit projection operators and formulas for integration of polynomials over the associated Stiefel manifold.
\end{abstract}
\keywords{Harmonic analysis, Howe dual pairs, transvector algebras.} \\ \noindent
\msc{42B35, 17B10, 30G35}
\section{Introduction}

In 1917 Ernst Fischer proved the following result, see \cite{Fi}: {\em given the fixed homogeneous polynomial $q(x)$ on $\mR^m$ and an arbitrary polynomial $P_k(x)$ of degree $k$, one can always decompose the latter as 
\begin{equation*}
 P_k(x) = Q_k(x) + q(x)R(x)\ ,
\end{equation*}
where $Q_k(x)$ is a homogeneous polynomial of degree $k$ belonging to the kernel of the differential operator $q(D)$, obtained by replacing each variable $x_j$ in $q(x)$ by the partial differential operator $\partial_{x_j}$, and where $R(x)$ has the appropriate degree.} 
\\ \\ \noindent
Generalisations of this result appear in a variety of mathematical contexts, a few of which are listed here:  
\begin{enumerate}
\item[(i)] In the framework of harmonic analysis, the subject in which the present paper is to be situated, the Fischer decomposition (FD) appears as the main tool for decomposing the space $\mcP(\mR^m,\mC)$ of polynomials on $\mR^m$ with values in the field $\mC$ of scalars into null-solutions for the Laplace operator. Similarly, the FD is used to decompose the space $\mcP(\mR^m,\mC_m)$ or $\mcP(\mR^m,\mS)$ of Clifford algebra-valued or spinor-valued polynomials on $\mR^m$ into null-solutions for the Dirac operator (which factorizes the Laplace operator). This serves as the starting point for several problems in multivariate analysis which generalise classical problems from complex analysis, since the FD in these particular cases should be interpreted as the higher-dimensional analogue of the Taylor series for holomorphic functions \cite{DSS}.
\item[(ii)] If the operator $q(D)$ enjoys certain invariance properties (e.g. the Laplace operator being invariant under the orthogonal group), the FD  can also be seen as an explicit realisation for an abstract decomposition of certain vector spaces into irreducible components. This subject is closely related to the existence of a Howe dual pair (see \cite{GW, Ho}), which will play a crucial role in the present paper (see section 4). For a nice overview of the interplay between Howe dual pairs and FDs in Clifford analysis, we refer to \cite{BDSES}. 
\item[(iii)] More specifically, explicit realisations of the FD on spaces of polynomials have been established for various other groups. For the case of finite reflection groups $G < O(m)$, this is achieved using Dunkl operators, see \cite{DX} for the harmonic version, as well as \cite{DBGV, HDB3, Orsted} for various Dirac versions. The case of the symplectic Dirac operator is treated in \cite{HDB4}. In the context of so-called hermitian Clifford analysis, the dual pair $ \mathfrak{sl}(1|2)\times U(n)$ is investigated \cite{BDS, BDSES}. Finally, in the context of flat superspaces, the Howe duality $\mathfrak{sl}(2)\times\mathfrak{osp}(m|2n)$ was studied in \cite{KC, KH}, while the Howe duality $\mathfrak{osp}(1|2)\times\mathfrak{osp}(m|2n)$ and corresponding FD was studied in \cite{HDB1}.
\item[(iv)] In a series of papers \cite{AGV, AR, Sha}, the FD was used to develop fast exact algorithms for solving Dirichlet boundary problems with polynomial data on quadratic surfaces on $\mR^m$. 
\end{enumerate}
In order to explain the problem addressed in the present paper, we briefly sketch the classical (harmonic) Fischer decomposition for scalar-valued polynomials on $\mR^m$.  This can be obtained by choosing $q(x) =  \norm{x}^2=\sum_{j}x_j^2$, the squared norm of the vector variable $x \in \mR^m$, for which $q(D) = \Delta_x$ becomes the Laplace operator on $\mR^m$. The harmonic Fischer decomposition then states that arbitrary polynomials $P_k(x)$, homogeneous of degree $k \in \mN$ in $x \in \mR^m$, can be decomposed into polynomials $H_{k-2j}(x)$ of degree $(k - 2j)$, where $0 \leq 2j \leq k$:
\begin{equation}\label{example_Fischer}
P_k(x) = \sum_{j = 0}^{\lfloor \frac{k}{2} \rfloor} \norm{x}^{2j}H_{k-2j}(x)\ , 
\end{equation}
with $\Delta_xH_{k-2j}(x) = 0$. From the point of view of representation theory for the Lie group $\operatorname{SO}(m)$ this result corresponds to the decomposition of the space $\mcP_k(\mR^m,\mC)$ of $k$-homogeneous $\mC$-valued polynomials into the subspaces of harmonics, irreducible under the regular $\operatorname{SO}(m)$-action
\begin{equation*}
\operatorname{H} : \SO(m) \rightarrow \Aut\big(\mcP(\mR^m,\mC)\big) : g \mapsto \operatorname{H}(g)
\end{equation*}
with $\operatorname{H}(g)[P](x) := P(g^{-1}x)$, where $\SO(m)$ acts on $x \in \mR^m$ through ordinary matrix multiplication. More generally, it gives rise to the full decomposition
\begin{equation} \label{Classic_Fischer}
\mcP(\mR^m,\mC) = \bigoplus_{k=0}^{+\infty}\mcP_k(\mR^m,\mC) \cong \bigoplus_{j=0}^{+\infty}\mV_{j}^\infty\otimes \mcH_j(\mR^m,\mC) \ , 
\end{equation}
where $\mV_{j}^\infty$ denotes the irreducible highest weight Verma module with highest weight $\mu = -(j + \frac{m}{2})$ for the (harmonic) dual partner $\spl(2)$ given by
\begin{equation}\label{harm_sl}
\spl(2) = \operatorname{Alg}(X,Y,H) \cong \operatorname{Alg}\left(-\frac{1}{2}\Delta_x,\frac{1}{2}\nx^2,-\left(\mE_x + \frac{m}{2}\right)\right)\ ,
\end{equation}
with $\mE_x:=\sum_{j=1}^m x_j\D_{x_j}$ the Euler operator, and where $\mcH_j(\mR^m,\mC)$ stands for the space of $j$-homogeneous polynomial solutions for the Laplace equation. In practical problems, such as integration over the sphere (or unit ball) in $\mR^m$ using Pizzetti's formula \cite{HDB1, HDB2, Pizzetti}, it is often useful to have an {explicit} expression for these harmonic components in $x \in \mR^m$. In case of harmonic analysis in one vector variable, these can be found using the projection operators from lemma \ref{projection_classic}.
\\ \\ \noindent
The aim of this paper is to obtain a generalisation of (\ref{example_Fischer}) and (\ref{Classic_Fischer}) for homogeneous polynomials $P_{k,l}(x,u)$ in two vector variables $(x,u) \in \mR^{2m}$. Polynomials in two vector variables can also be decomposed into irreducible components, corresponding to the decomposition of the vector space $\mcP_{k,l}(\mR^{2m},\mR)$ into irreducible summands for the regular action of $\SO(m)$ by means of
\begin{equation*}
\operatorname{H}(g)[P](x,u) := P(g^{-1}x,g^{-1}u)\ . 
\end{equation*} 
This is described by the Howe dual pair $\SO(m) \times \sym(4)$, a classical result in representation theory for which we refer to e.g
\cite{G, Ho, Hrem} and theorem \ref{theorem_Howe}. However, despite the fact that this result is well understood from a theoretical 
point of view, it does not explain how to explicitly obtain the summands inside the decomposition. To achieve this, we will propose a 
new decomposition: first the classical Fischer decomposition is used in the variables  $x\in\mR^{m}$ and $u\in \mR^{m}$ 
separately to obtain polynomials which are harmonic in $x$ and $u$, i.e. belong	to $\ker \brac{\Delta_x, \Delta_u}$. 
This reduces the problem to decomposing such polynomials into irreducible summands for $\SO(m)$. 
To do so, we introduce a so-called transvector algebra. This is a particular subalgebra of $\End\brac{\ker \brac{\Delta_x, \Delta_u}}$, which 
generalises the role played by the Lie algebra $\sym(4)$ as the dual partner in the classical Howe duality.  
Not only does this allow us to define an analogue for the Verma $\mV_k^\infty$ appearing in the aforementioned approach (see theorem \ref{Main-theorem}, the first main result of the paper), but also the quadratic relations satisfied by the generators can be employed to define explicit projection operators on the irreducible summands appearing in the decomposition (see theorems \ref{main-projection} and \ref{Projection_operator_2}). 
Moreover, in terms of these operators we can obtain a Pizzetti formula for the integration for polynomials over the Stiefel manifold, and they are central in our proof for the orthogonality of the summands appearing in the decomposition into irreducible subspaces with respect to the classical Fischer decomposition. 
The main advantage of our approach over the classical Howe duality approach lies in the fact that we reduce the analysis from the full space of polynomials to the subspace $\ker \brac{\Delta_x, \Delta_u}$. This is reflected in the projection operator lying at the core of our arguments (an operator which somehow generalises the role of the extremal projection operator for a classical Lie algebra, see \cite{Zh3}), which is considerably less complicated to handle (in the sense that it contains a double summation, rather than a fourfold summation). 
\\ \\ \noindent
The paper is organised as follows. After discussing the classical Howe duality on $\mR^{2m}$ in section \ref{section Howe}, we will introduce the harmonic transvector algebra in section \ref{transvector}. In section \ref{module}, we will define and study the module $\mV\comm{k,l}$ for this algebra, which is the analogue of $\mV_k^\infty$. We will also construct explicit projections on the irreducible summands. In section \ref{example}, an explicit example will be considered and in section \ref{orthogonality}, the orthogonality of the different irreducible summands with respect to the Fischer inner product will be proved. Finally, in the last section, we will prove a Pizzetti type formula for integration of polynomials over the Stiefel manifold invoking the projection operators obtained in section \ref{module}. The proof of theorem \ref{main-projection} will be gathered in the appendix. 

\section*{Conventions}
Throughout the paper, we work in dimensions greater than four, i.e. $m>4$, to ensure that the Lie algebra $\mathfrak{so}(m)$ is simple. Also, all Lie algebras appearing in the paper are over the field of complex numbers. 

\section*{Acknowledgement}
This research was supported by the Fund for Scientific Research-Flanders (FWO-V), project ``Construction of algebra realisations using Dirac-operators'', grant G.0116.13N. The authors would also like to thank Kevin Coulembier, Tim Janssens, Roy Oste and Joris Van der Jeugt for their useful suggestions which have led to a considerable improvement of the paper.


\section{Classical Howe duality on $\mR^{2m}$}
\label{section Howe}
In this section we will introduce the complete decomposition for the space $\mcP(\mR^{2m},\mC)$ of polynomials in two variables $(x,u)\in \mR^{2m}$, using the Howe dual pair $\sym(4)\times\textup{SO}(m)$. The notation $\sym(4) = C_2$ hereby refers to the symplectic Lie algebra of rank 2. As mentioned earlier, the vector space $\mcP(\mR^{2m},\mC)$ carries the structure of an SO$(m)$-module under the regular representation $H$. The Weyl algebra in two vector variables is given by: 
\begin{equation*}
\mcW(\mR^{2m},\mC) = \operatorname{Alg}\big(x_i,u_j;\partial_{x_p},\partial_{u_q}\big)_{1 \leq i, j, p, q \leq m}\,
\end{equation*}
and the action of the group SO$(m)$ on this space is for all $D \in \mcW(\mR^{2m},\mC)$ defined by means of $g \cdot D := \textup{H}(g) \circ D \circ \operatorname{H}(g^{-1})$. A well-known result from classical representation theory shows us that the following subspaces generate all SO$(m)$-invariant operators (see theorem 5.6.9 on page 285 of \cite{GW}): 
\begin{align*}
\p^+ &:=  \Span(\Delta_x,\inner{\D_u,\D_x}, \Delta_u)\\
\p^- &:= \Span(\nx^2,\inner{u,x}, \nuu^2)\\
\gok^{\phantom{-}} &:= \Span(\inner{x,\D_u}, \inner{u,\D_x}, \mE_x + \frac{m}{2}, \mE_u + \frac{m}{2})\ .
\end{align*}
The notation $\langle\cdot,\cdot\rangle$ hereby refers to the Euclidean inner product on $\mR^m$, and an operator such as for example $\inner{x,\D_u}$ then stands for $\sum_jx_j\partial_{u_j}$. It is easily verified that the sum $\g := \p^+ \oplus \gok \oplus \p^-$ defines a model for the Lie algebra $\sym(4)$, where $\gok$ realises the reductive Lie algebra $\gl(2)$. The Lie algebra $\g$ is then called the Howe dual partner for the action of the Lie group SO$(m)$ on polynomials in $(x,u) \in \mR^{2m}$. This is the subject of Howe's celebrated result, for which we refer e.g. to \cite{G,GW,Ho,Hrem}. We state the main conclusion here, which will be exploited in what follows:
\begin{theorem}\label{theorem_Howe}
Under the joint action of $\sym(4) \times \mathrm{SO}(m)$, the space $\mcP(\mR^{2m},\mC)$ decomposes as follows:
\begin{equation}\label{Duality}
\mcP(\mR^{2m},\mC) \cong \bigoplus_{k \geq l} \mV^\infty_{k,l} \otimes \mcH_{k,l}(\mR^{2m},\mC)\ . 
\end{equation}
The summation is hereby performed over all pairs of integers $(k,l) \in \mZ^+\times\mZ^+$ satisfying the dominant weight condition $k \geq l \geq 0$ and
the vector spaces at the right-hand side are defined as:
\begin{itemize}
\item[(i)] the infinite-dimensional (irreducible) $\sym(4)$-module
\begin{equation*}
\mV^\infty_{k,l} := \mcU(\g) \otimes_{\mcU(\gok\oplus \p^+)}\mV_{k,l},
\end{equation*}
which is induced from the (finite-dimensional) irreducible $\gl(2)$-module $\mV_{k,l}$, with highest weight given by $\mu = (k+\frac{m}{2},l+\frac{m}{2})$, extended to a $\gok\oplus \p^+$-module by letting $\p^+$ act trivially on $\mV_{k,l}$,
\item[(ii)] the space $\mcH_{k,l}$, with $k\geq l$, defined by means of
\begin{equation*}
\mcH_{k,l}(\mR^{2m},\mC) := \mcP_{k,l}(\mR^{2m},\mC) \cap \ker\big(\Delta_x,\Delta_u,\inner{\D_u,\D_x},\inner{x,\D_u}\big)\ ,
\end{equation*}
with $\ker(D_1,\ldots,D_k)=\ker D_1 \cap \ldots \cap \ker D_k$. 
\end{itemize}
\end{theorem}
\noindent
Note that the vector space $\mcH_{k,l}(\mR^{2m},\mC)$, of so-called simplicial harmonics, defines a model for the irreducible $\SO(m)$-module with highest weight $(k,l,0,\ldots,0)$ for $k \geq l$, see e.g. \cite{CSV,GM}. This generalises a well-known classical result which says that the vector space $\mcH_k(\mR^m,\mC)$ of $k$-homogeneous harmonic polynomials defines a model for the irreducible representation with highest weight $(k,0,\cdots,0)$. Also note that the module $\mV^\infty_{k,l}$ is freely generated by $\mcU(\p^-)$ and as such, we have $\mV^\infty_{k,l}\cong\mcU(\p^-)\cdot\mV_{k,l}$. The isomorphism (\ref{Duality}) has to be interpreted in the following sense: for every pair of integers $(k,l)$, there is an injective $(\sym(4)\times \mathrm{SO}(m))$-intertwining map 
\begin{equation*}
\begin{split}
\mV^\infty_{k,l} \otimes \mcH_{k,l}(\mR^{2m},\mC) &\longrightarrow \mcP(\mR^{2m},\mC) \\
\Phi(Y_1,Y_2,Y_3)Y^j w_{k,l}\otimes H_{k,l}(x,u) &\mapsto \Phi(\nx^2,\inner{u,x}, \nuu^2)\inner{u,\D_x}^jH_{k,l}(x,u),
\end{split}
\end{equation*}
where $w_{k,l}$ is the highest weight vector for the module $\mV_{k,l}$ from above, $Y$ is the negative root vector of $\gl(2)$ and $\Phi(Y_1,Y_2,Y_3)$ is a polynomial in the variables $\set{Y_1,Y_2,Y_3}$ that span $\p^-$. This means that the representation $\mV_{k,l}$ can be realised explicitly as
\begin{equation*} 
\mV_{k,l} = \bigoplus_{j=0}^{k-l}\operatorname{Span}\Big(\inner{u,\D_x}^jH_{k,l}(x,u)\Big),
\end{equation*}
for an arbitrary but fixed $H_{k,l}(x,u)\in \mcH_{k,l}(\mR^{2m},\mC)$. As a consequence of Howe's abstract result, one can decompose arbitrary (homogeneous) polynomials in $(x,u)$ as follows: 
\begin{equation}\label{decomp1}
P_{k,l}(x,u) = \sum_{a,b,j,p,q} \rho_{a,b}(x,u)\inner{u,\D_x}^jH_{p,q}(x,u)\ , 
\end{equation}
with $\rho_{a,b}(x,u)\in \mcU(\p^-)$ and $\operatorname{deg}_x(\rho_{a,b}(x,u))=a$, $\operatorname{deg}_u(\rho_{a,b}(x,u))=b$. The summation is such that $(a + p - j,b + q + j) = (k,l)$. 
\\ \\ \noindent
However, this is not the approach we will pursue in this paper. Instead, as an alternative for the decomposition (\ref{decomp1}) above, we start from the following decomposition: 
\begin{equation}\label{decomp2}
P_{k,l}(x,u) = \sum_{i = 0}^\kappa\sum_{j=0}^\lambda \nx^{2i}\nuu^{2j}H'_{k-2i,l-2j}(x,u)\ ,
\end{equation}
where $\kappa = \lfloor\frac{k}{2}\rfloor$ and $\lambda = \lfloor\frac{l}{2}\rfloor$, and with 
\begin{equation*} 
H'_{k-2i,l-2j}(x,u) \in \mcP_{k-2i,l-2j}(\mR^{2m},\mC) \cap \ker\big(\Delta_x,\Delta_u\big)\ . 
\end{equation*}
It is crucial to note that the double harmonic vector space in the previous formula is \textit{not} irreducible under the action of the group SO$(m)$, which is why we have added an extra prime to the notation. It is then clear that as a module for SO$(m)$ we have: 
\begin{equation*}
 \mcP_{a,b}(\mR^{2m},\mC) \cap \ker\big(\Delta_x,\Delta_u\big) \cong (a,0,\cdots,0) \otimes (b,0,\cdots,0)\ . 
 \end{equation*}
The tensor product at the right-hand side is well-understood (note that from now on we will omit redundant zeroes in the highest weights): 
\begin{theorem}(Klymik, \cite{K})\label{Klymik}
For all positive integers $a \geq b \in \mZ^+$, one has the following decomposition: 
\begin{equation*}
(a) \otimes (b) \cong \bigoplus_{i=0}^b\bigoplus_{j=0}^{b-i}(a-i+j,b-i-j)\ . 
\end{equation*}
\end{theorem}
\noindent
This theorem thus says that the product in $\mcP(\mR^{2m},\mC)$ of two harmonic polynomials (resp. in $x$ and $u$) can be decomposed in terms of simplicial harmonic polynomials (see Theorem \ref{theorem_Howe}). It is then clear that the problem of decomposing arbitrary polynomials in two vector variables can be solved if one characterises the following operators: 
\begin{itemize}
\item[(i)] The projection operators 
\begin{equation*}
\pi_{i,j} : \mcP_{p,q}(\mR^{2m},\mC) \longrightarrow \mcP_{p-2i,q-2j}(\mR^{2m},\mC) \cap \ker\big(\Delta_x,\Delta_u\big)\ . 
\end{equation*}
As we will show in the next section, this can be done using the extremal projection operator for the Lie algebra $\spl(2)$. This is the analogue of the classical Fischer decomposition in only one vector variable (see the introduction). 
\item[(ii)] The projection operators
\begin{equation*}
\Pi_{i,j} : \mcP_{k,l}(\mR^{2m},\mC) \cap \ker\big(\Delta_x,\Delta_u\big) \longrightarrow \mcH_{k-i+j,l-i-j}(\mR^{2m},\mC)\ . 
\end{equation*}
This requires considerably more effort, and will be done in terms of a particular infinite-dimensional algebra which is known in the literature as a transvector algebra. 
\end{itemize}
\begin{remark}
Throughout this paper, the term `projection operator' is used somewhat loosely in the sense that the projection operators we will use do not always act on one and the same space and are as such only projection operators `up to isomorphism'.
\end{remark}


\section{The double harmonic transvector algebra}
\label{transvector}
In this section, we start with a particular decomposition for the algebra $\g$, which is closer to the spirit of the alternative decomposition (\ref{decomp2}): 
\begin{equation*}
\g=\sym(4) := \gs \oplus \gt = \big(\spl_x(2) \oplus \spl_u(2)\big) \oplus \gt \cong \so(4) \oplus \gt\ , 
\end{equation*}
with $\spl(2)$ the Lie algebra (\ref{harm_sl}) realised in harmonic analysis (in both variables $x$ and $u$, whence the extra subscript) and $\gt \subset \sym(4)$ the subspace defined as
\begin{equation*}
\gt := \mbox{Span}\big(\inner{u,x}, \inner{\D_u,\D_x}, \inner{u,\D_x}, \inner{x,\D_u}\big)\ . 
\end{equation*}
Note that $\gt \subset \g$ is not a subalgebra, but it clearly defines an $\gs$-module under the adjoint (commutator) action. Also note that there is no other algebra of two commuting copies of $\spl(2)$ with these properties. We then define the set
\begin{equation*}
\gs^+ := \left\{-\frac{1}{2}\Delta_x,-\frac{1}{2}\Delta_u\right\} \subset \gs\ , 
\end{equation*}
containing the positive root vectors inside the semi-simple Lie algebra $\gs$, and introduce a left-sided ideal $J' := \mcU'(\g)\gs^+$ in $\mcU'(\g)$. The prime in $\mcU'(\g)$ denotes the localisation of the universal enveloping algebra $\mcU(\g)$ with respect to the subset $\mcU(\goh)$, where $\goh$ stands for the abelian Cartan algebra in $\g$, explicitly given by 
\begin{equation*}
\goh := \Span(H_x,H_u) = \operatorname{Span}\left(-\left(\mE_x + \frac{m}{2}\right),-\left(\mE_u + \frac{m}{2}\right)\right)\ .
\end{equation*}
Defining the normaliser
\begin{equation*}
\operatorname{Norm}(J') := \big\{u \in \mcU'(\g) : J'u \subset J'\big\}\ , 
\end{equation*}
we finally arrive at the algebra we are interested in:
\begin{definition}
Recalling the notation $\g=\sym(4) = \gs \oplus \gt$ from above, we introduce the transvector algebra
\begin{equation*}
Z\big(\g,\gs\big) = \operatorname{Norm}(J')/J'\ ,
\end{equation*}
with $J' = \mcU'(\g)\gs^+$. 
\end{definition}
\noindent
This algebra was first introduced by Zhelobenko, see \cite{Zh}, in the framework of describing solutions for equations (so-called extremal systems, such as the Dirac or Maxwell equations) through the characterisation of the algebra of symmetries (the transvector algebra). Due to the works of Zhelobenko and Molev \cite{Molev, Zh, Zh2, Zh3}, we know that the algebra $Z\big(\g,\gs\big)$ is generated by the elements $\pi_\gs[u]$, with $u \in \gt$ and $\pi_\gs \in \mcU'(\gs)$ the extremal projection operator for $\gs$. Since the algebra $\gs$ consists of two commuting copies of $\spl(2)$, we will first define the extremal projection operator for $\spl(2)$:
\begin{definition}\label{pi}
The extremal projection operator $\pi$ for the Lie algebra $\spl(2)$ is defined as the unique formal operator, contained in an extension of $\mcU'\big(\spl(2)\big)$ to some algebra of formal series, which satisfies the conditions $X \pi = \pi Y = 0$ and $\pi^2 = \pi$. 
\end{definition}
\noindent
One can prove that this operator is explicitly given by
\begin{equation*}
\pi = 1+ \sum_{j = 1}^\infty\frac{(-1)^j}{j!}\frac{\Gamma(H + 2)}{\Gamma(H + j + 2)}Y^j X^j\ .
\end{equation*}
The localisation is needed to obtain a true projection operator ($\pi^2=\pi$), and is reflected in the occurrence of the operators $H \in \goh$ in the denominator. 
\begin{remark}
Note that we will often encounter expressions of the form $A/B$ in what follows, where in general $A$ and $B$ stand for non-commuting operators. Fractions like these will always stand for the product $B^{-1}A$, in this particular order. 
\end{remark}
\noindent
This extremal projection operator thus projects an arbitrary element $v_\lambda$ of an (irreducible, not necessarily finite-dimensional) representation $\mV_\lambda$ for $\spl(2)$ onto a multiple $\pi[v_\lambda]$ of the highest weight vector for $\mV_\lambda$. Note that this operator exists for all simple Lie algebras and superalgebras, see e.g. the overview paper \cite{T} and the references mentioned therein. In particular, we can use this extremal projection operator $\pi_x$ for the harmonic realisation (\ref{harm_sl}), in the sense that $\pi_x[P_k(x)] = H_k(x)$ immediately gives the harmonic component. The projection on any other component can also be written down in terms of the operator $\pi_x$: 
\begin{lemma}\label{projection_classic}
For all $s \in \mZ^+$, the projection operator
\begin{equation*}
\pi^{(s)} : \mcP(\mR^m,\mC) \rightarrow \nx^{2s}\mcH(\mR^m,\mC)\ ,
\end{equation*}
is explicitly given by
\begin{equation*}
P(x) \mapsto \pi^{(s)}[P(x)] := \frac{1}{4^ss!}\frac{\Gamma\left(\mE_x + \frac{m}{2} - 2s\right)}{\Gamma\left(\mE_x + \frac{m}{2} - s\right)}\nx^{2s}\pi_x\Delta_x^s P(x)\ .
\end{equation*}
Note that we hereby obviously have that $\pi^{(0)} = \pi_x$. 
\end{lemma}
\begin{proof}
It suffices to note that 
\begin{equation*}
\Delta_x^s\nx^{2s}H_k(x) = 4^ss!\frac{\Gamma\left(k + \frac{m}{2} + s\right)}{\Gamma\left(k + \frac{m}{2}\right)}H_k(x)\ .
\end{equation*}
The numerical constant can be formulated in terms of the Euler operator, giving rise to the quotient of Gamma functions appearing at the right-hand side of the expression above. \end{proof}
\noindent
In view of the fact that the generators of the Lie algebras $\spl(2)_x$ and $\spl(2)_u$ commute, we then have that $\pi_\gs$ is given by the composition of the associated extremal projection operators for $\spl(2)$ from definition \ref{pi}, denoted by $\pi_x$ and $\pi_u$. Our operator $\pi_\gs$ is thus given by $\pi_\gs = \pi_x\pi_u = \pi_u\pi_x$. More explicitly, we have: 
\begin{equation*}
\pi_\gs   =  \left(\sum_{j=0}^\infty\frac{1}{4^jj!}\frac{\Gamma(H_x + 2)}{\Gamma(H_x + 2 + j)}\nx^{2j}\Delta_x^j \right)\left(\sum_{j=0}^\infty\frac{1}{4^jj!}\frac{\Gamma(H_u + 2)}{\Gamma(H_u + 2 + j)}\nuu^{2j}\Delta_u^j \right) ,
\end{equation*}
and this operator satisfies the relations $\pi_\gs^2 = \pi_\gs$ and
\begin{equation*}
\Delta_x\pi_\gs = \Delta_u\pi_\gs = 0 = \pi_\gs\nx^2 = \pi_\gs\nuu^2\ .
\end{equation*}
As mentioned before, the projection $\pi_{k,l}^{i,j}$ can be done using the extremal projection operator. This is indicated in the following theorem:
\begin{theorem}\label{Projection_operator_1}
The projection operator
\begin{equation*}
\pi_{i,j} : \mcP_{p,q}(\mR^{2m},\mC) \longrightarrow \mcP_{p-2i,q-2j}(\mR^{2m},\mC) \cap \ker\big(\Delta_x,\Delta_u\big)\ ,
\end{equation*}
is given by
\begin{equation*}
\pi_{i,j}\comm{P_{p,q}(x,u)}:=\frac{1}{4^{i+j} i!j!}\frac{\Gamma\left(\mE_x + \frac{m}{2} \right)}{\Gamma\left(\mE_x + \frac{m}{2}+i \right)}\frac{\Gamma\left(\mE_u + \frac{m}{2} \right)}{\Gamma\left(\mE_u + \frac{m}{2}+j\right)}\pi_{\gs}\Delta_x^i\Delta_u^j P_{p,q}(x,u)\ .
\end{equation*}
\end{theorem}
\noindent
Next, we will label the algebraic generators for the transvector algebra $Z\big(\g,\gs\big)$ as follows: 
\begin{definition}\label{Generators}
The set $\mcG \subset \End(\ker \Delta_x \cap \ker \Delta_u)$, containing the generators for the algebra $Z\big(\g,\gs\big)$, is defined as: 
\begin{equation*} 
\mcG := \big\{S_x,S_u;A,C\big\} = \big\{\pi_\gs\inner{x,\D_u},\pi_\gs\inner{u,\D_x};\pi_\gs\inner{\D_u,\D_x},\pi_\gs\inner{u,x}\big\}\ . 
\end{equation*}
\end{definition}
\begin{remark}
The names of these generators are not chosen randomly: $S_x$ and $S_u$ both contain a so-called `skew' operator (combining a variable, indicated by the subscript, and a differential operator), and $A$ (respectively $C$) will be interpreted as an annihilation (respectively creation) operator, see below. 
\end{remark}
\noindent
When plugging in the explicit form for the operator $\pi_\gs$ into $\pi_\gs[u]$, with $u \in \gt$, the infinite series reduces to a finite sum since $\pi_\gs[u]$ must act on double harmonic polynomials (see definition \ref{Generators}). For the generators of $Z\big(\g,\gs\big)$, we then obtain the following expressions (recall that $A/B = B^{-1}A$): 
\begin{align*}
S_x & = \inner{x,\D_u} - \frac{|x|^2\inner{\D_u,\D_x}}{2\mE_x + m - 4}\\
S_u & =  \inner{u,\D_x} - \frac{|u|^2\inner{\D_u,\D_x}}{2\mE_u + m - 4}\\
A & =  \inner{\D_u,\D_x}\\
C & =  \inner{u,x} - \frac{|x|^2\inner{u,\D_x}}{2\mE_x + m - 4} - \frac{|u|^2\inner{x,\D_u}}{2\mE_u + m - 4} + \frac{|x|^2|u|^2\inner{\D_u,\D_x}}{(2\mE_x + m - 4)(2\mE_u + m - 4)}\ .
\end{align*}
Note that the `skew' operators only have two terms. This is due to the fact that for example {$\Delta_x^2\inner{x,\D_u} H(x,u)=0$} and $\Delta_u^2\inner{x,\D_u} H(x,u)=0$, where $H(x,u)\in\ker \Delta_x \cap \ker \Delta_u$. The annihilation operator $A$ has only one term since $\comm{\Delta_x,\inner{\D_u,\D_x}}=0=\comm{\Delta_u,\inner{\D_u,\D_x}}$. \\ \\ \noindent
According to a general result \cite[Theorem 2]{Zh}, the generators of $Z\big(\g,\gs\big)$ must satisfy a set of quadratic relations: denoting $z_j = \pi[u_j]$ with $u_j \in \gt$ (with $\g=\gs \oplus \gt$ a Lie algebra and $\pi$ the extremal projection operator associated to the reductive part $\gs$), one has that 
\begin{equation} \label{quadratic_relations}
[z_a,z_b] = \sum_{k,l}\alpha_{a,b}^{k,l}z_kz_l + \sum_{k}\beta_{a,b}^kz_k + \gamma_{a,b}\ , 
\end{equation}
where all `coefficients' belong to $\mcU'(\goh)$, with $\goh \subset \g$ a Cartan algebra. Let us then explicitly derive these relations in our present context. Before we do so, we first mention some basic identities which we need throughout this section:
\begin{lemma}
For all integers $a \in \mN$, we have: 
\begin{alignat*}{3}
&[\inner{\D_u,\D_x},\nx^{2a}] =2a\nx^{2a-2}\inner{x,\D_u} \qquad && [\inner{x,\D_u},\Delta_x^{a}] = -2a\Delta_x^{a-1}\inner{\D_u,\D_x}\\
&[\inner{\D_u,\D_x},\nuu^{2a}]  =2a\nuu^{2a-2}\inner{u,\D_x} \qquad && [\inner{u,\D_x},\Delta_u^{a}] = -2a\Delta_u^{a-1}\inner{\D_u,\D_x} .
\end{alignat*}
\end{lemma}
\noindent
We start with a collection of quadratic relations which can be interpreted as `scaled' commutation relations: 
\begin{lemma}\label{scaled_commutation}
The following relations hold in $Z\big(\g,\gs\big)$: 
\begin{align}
A S_x & =  \frac{H_x + 2}{H_x + 1}\:S_x A\label{A_Sx} \\
A S_u & =  \frac{H_u + 2}{H_u + 1}\:S_u A \label{A_Su}\\
S_uC & = \frac{H_x + 2}{H_x + 1}\:C S_u \label{C_Su}\\
S_xC & = \frac{H_u + 2}{H_u + 1}\:C S_x \label{C_Sx}\ .
\end{align}
\end{lemma}
\noindent
\begin{proof}
We will prove the first of these relations, the other relations are then proved in a similar way (or follow by symmetry). By definition, we have:
\begin{equation*}
\pi_\gs\inner{x,\D_u}\pi_\gs\inner{\D_u,\D_x}  =  \pi_\gs\inner{x,\D_u}\inner{\D_u,\D_x}\ .
\end{equation*}
On the other hand, in view of the fact that $\Delta_u$ commutes with $\inner{x,\D_u}$, we also have that 
\begin{align*}
\pi_\gs\inner{\D_u,\D_x}\pi_\gs\inner{x,\D_u} & =  \pi_\gs\inner{\D_u,\D_x}\pi_x\inner{x,\D_u}\\
& =  \pi_\gs\inner{\D_u,\D_x}\left(1 + \frac{1}{4(H_x + 2)}\nx^2\Delta_x + \mathcal{O}(\Delta_x^2)\right)\inner{x,\D_u}\\
& =  \pi_\gs\inner{\D_u,\D_x}\inner{x,\D_u} + \frac{1}{H_x + 1}\pi_\gs\inner{x,\D_u}\inner{\D_u,\D_x}\\
& =  \left(1 + \frac{1}{H_x + 1}\right)\pi_\gs\inner{x,\D_u}\inner{\D_u,\D_x}\ .
\end{align*}
This leads to the desired relation. 
\end{proof}
\begin{remark}
The relations derived in lemma \ref{scaled_commutation} are not of the form (\ref{quadratic_relations}), but turning them into such is a form is straightforward and not necessary for our purpose. 
\end{remark}
\begin{remark}
Note that the relations above can be rewritten as relations for commuting operators, but this requires rescaling them. The resulting operators are then said to belong to the Mickelsson algebra, see \cite{Mi, Molev}. For example, defining the rescaled operator 
\begin{equation*}
\widetilde{S}_x := (2\mE_x + m - 4)\inner{x,\D_u} - |x|^2\inner{\D_u,\D_x}\ ,
\end{equation*}
it is easily verified that the relation $[\widetilde{S}_x,A] = 0$ holds. Similar relations can then be obtained for the other identities from lemma \ref{scaled_commutation}. 
\end{remark}
\noindent
Combining the operators $S_x$ and $S_u$, we get the following result: 
\begin{lemma}\label{lemmaSx_Su}
The following commutation relation holds in $Z\big(\g,\gs\big)$: 
\begin{equation}\label{Sx_Su}
[S_x,S_u]  =  \frac{H_x - H_u}{(1 + H_x)(1 + H_u)}CA - (H_x - H_u)\ . 
\end{equation}
\end{lemma}
\begin{proof}
Using the fact that $\nx^2$ commutes with $\inner{x,\D_u}$, we get: 
\begin{align*}
S_xS_u & =  \pi_\gs\inner{x,\D_u}\pi_\gs\inner{u,\D_x}\\
& = \pi_\gs\inner{x,\D_u}\left(1 + \frac{1}{4(H_u + 2)}\nuu^2\Delta_u + \mathcal{O}(\Delta_u^2)\right)\inner{u,\D_x}\\
& = \pi_\gs\inner{x,\D_u}\inner{u,\D_x} + \frac{1}{H_u + 1}\:\pi_\gs\inner{u,x}\inner{\D_u,\D_x}\ .
\end{align*}
Swapping $x$ and $u$, we immediately obtain: 
\begin{equation*}
S_uS_x = \pi_\gs\inner{u,\D_x}\inner{x,\D_u} + \frac{1}{H_x + 1}\:\pi_\gs\inner{u,x}\inner{\D_u,\D_x}\ .
\end{equation*}
Together with the fact that $CA = \pi_\gs\inner{u,x}\pi_\gs\inner{\D_u,\D_x} = \pi_\gs\inner{u,x}\inner{\D_u,\D_x}$, this easily leads to the desired relation. 
\end{proof}
\noindent
Finally, combining the operators $C$ and $A$, we arrive at the complicated relation: 
\begin{lemma}\label{AC}
The following relation holds in $Z\big(\g,\gs\big)$: 
\begin{equation*}
AC  =  \frac{H_x + H_xH_u + H_u}{(H_x+ 1)(H_u + 1)}CA - (H_x + H_u) +  \frac{S_xS_u}{H_x+1} + \frac{S_uS_x}{H_u+1}\ .
\end{equation*}
\end{lemma}
\begin{proof}
By definition, we have: 
\begin{align*}
AC & =  \pi_\gs\inner{\D_u,\D_x}\left(1 + \frac{1}{4(H_x + 2)}\nx^2\Delta_x + \mathcal{O}(\Delta_x^2)\right)\pi_u\inner{u,x}\\
& =  \pi_\gs\inner{\D_u,\D_x}\pi_u\inner{u,x} + \frac{1}{H_x + 1}\pi_\gs\inner{x,\D_u}\pi_u\inner{u,\D_x}\ .
\end{align*}
The first term can be simplified as follows: 
\begin{align*}
\pi_\gs\inner{\D_u,\D_x}\pi_u\inner{u,x} & =  \pi_\gs\inner{\D_u,\D_x}\left(1 + \frac{1}{4(H_u + 2)}\nuu^2\Delta_u + \mathcal{O}(\Delta_u^2)\right)\inner{u,x}\\
& =  \pi_\gs\inner{\D_u,\D_x}\inner{u,x} + \frac{1}{H_u + 1}\:\pi_\gs\inner{u,\D_x}\inner{x,\D_u}\ .
\end{align*}
The second term from above can also be simplified, using the fact that
\begin{align*}
\pi_\gs\inner{x,\D_u}\pi_u\inner{u,\D_x}
& =  \pi_\gs\inner{x,\D_u}\inner{u,\D_x} + \frac{1}{H_u+1}\pi_\gs\inner{u,x}\inner{\D_u,\D_x}\ .
\end{align*}
Using the expressions for $S_xS_u$ and $S_uS_x$ from the proof of lemma \ref{lemmaSx_Su}, we then arrive at the desired result.
\end{proof}
\noindent


\section{Simplicial modules for the algebra $Z\big(\g,\gs\big)$} \label{module}
\subsection{Description of the module}
We will now prove that the action of the generators in $\mcG$ creates, for all pairs of integers $k \geq l \geq 0$, an infinite-dimensional module $\mV[k,l]$. For that purpose, we can consider the space $\mcH_{k,l}$ of simplicial harmonics as a highest weight vector, since it is killed under the action of the operators $A$ and $S_x$. Moreover, we also have that $\pi_\gs[\goh] = \goh$ and this Cartan algebra then reproduces the highest weight. 
\begin{lemma}
For each pair of integers $k \geq l \geq 0$, we have that $S_u^{k-l+1}\mcH_{k,l} = 0$, whereas the action of $S_u^j$ is non-trivial for all integers $j \leq k-l$. 
\end{lemma}
\begin{proof}
We will make use of the fact that the generators for $Z\big(\g,\gs\big)$ are invariant under the action of SO$(m)$, as this allows us to consider the action of $S_{u}$ on the highest weight vector for $\mcH_{k,l}$ only. It was shown e.g. in \cite{CSV} that this particular element is given by
\begin{equation*}
H^{(h)}_{k,l}(x,u) := \overline{(z_1)^{k-l}(z_1w_2- z_2w_1)^l} \ ,
\end{equation*}
where the complex variables are defined as $z_j = x_{2j-1} + ix_{2j}$ and $w_k = u_{2k-1} + iu_{2k}$. It is now obvious that 
\begin{equation*}
 S_u^{k-l}H^{(h)}_{k,l}(x,u) =\inner{u,\D_x}^{k-l}H^{(h)}_{k,l}(x,u) = (-1)^{l}(k-l)!H^{(h)}_{k,l}(u,x)\ , 
\end{equation*}
which means that the operator $\inner{u,\D_x}^{k-l+1}$ is the smallest integer power which acts trivially on the vector space $\mcH_{k,l}$. We still need to prove that the projection $\pi_\gs$ on the space $\mcH'(\mR^{2m},\mC)=\mcP(\mR^{2m},\mC)\cap \ker\big(\Delta_x,\Delta_u\big)$ of double harmonics must be non-trivial for all smaller powers. First of all, as $[\Delta_x,\inner{u,\D_x}] = 0$ we have that $\pi_\gs\inner{u,\D_x} = \pi_u\inner{u,\D_x}$. Secondly, suppose that at some point the projection on the harmonics would become trivial. This would imply that the image under the operator $\inner{u,\D_x}^j$ belongs to the orthogonal complement $\nuu^2\mcH'(\mR^{2m},\mC)$ of the harmonics in $u$. However, as the operator $\inner{u,\D_x}$ commutes with this factor $\nuu^2$, this would mean that all consecutive images under the action of $\inner{u,\D_x}$ also belong to this orthogonal complement. In particular, we would then find that 
\begin{equation*}
\inner{u,\D_x}^{k-l}H^{(h)}_{k,l}(x,u) \in \nuu^2\mcH'(\mR^{2m},\mC)\ ,  
\end{equation*}
which is clearly a contradiction. 
\end{proof}
\begin{definition}\label{V_A[k,l]}
For all pairs of integers $k \geq l \geq 0$, the $Z\big(\g,\gs\big)\:$-module $\mV[k,l]$ can be realised by means of
\begin{equation*} 
\mV[k,l] := \bigoplus_{i=0}^{\infty}\bigoplus_{j=0}^{k-l}C^iS_u^jH_{k,l}(x,u)\ , 
\end{equation*}
for an arbitrary but fixed $H_{k,l}(x,u)\in \mcH_{k,l}\brac{\mR^{2m},\mC}$.
\end{definition}
\noindent
Note that the order of the operators is unimportant, in view of lemma \ref{scaled_commutation}. This is also reflected in the following:
\begin{proposition}\label{(C,S_u)-action}
The operators $(C,S_u)$ act as follows on the module $\mV[k,l]$: 
\begin{align*}
C : \mV[k,l] \rightarrow \mV[k,l] & :  C^iS_u^j\mcH_{k,l}\ \mapsto\ C^{i+1}S_u^j\mcH_{k,l}\ , \\
S_u : \mV[k,l] \rightarrow \mV[k,l] & :  C^iS_u^j\mcH_{k,l}\ \mapsto\ \frac{H_x + i + 1}{H_x + 1}\:C^{i}S_u^{j+1}\mcH_{k,l}\ . 
\end{align*}
\end{proposition}
\begin{proof}
The first statement is trivial, the second follows immediately from a repeated application of relation (\ref{C_Su}). 
\end{proof}
\noindent
Under the action of these two $Z\big(\g,\gs\big)\:$-generators $(C,S_u)$ on $\mcH_{k,l}$ we then obtain an infinite-dimensional strip. Note that this gives an explicit realisation for the way in which $\mcH_{k,l}$ is embedded in all possible tensor products $(a) \otimes (b)$ containing $(k,l)$ as a submodule, according to theorem \ref{Klymik}: 
\begin{equation*} 
C^iS_u^j : \mcH_{k,l}(\mR^{2m},\mC)\hookrightarrow \mcP_{k+i-j,l+i+j}(\mR^{2m},\mC) \cap \ker(\Delta_x,\Delta_u)\ . 
\end{equation*}
This means that we could give an alternative picture for the module $\mV[k,l]$, in which each subspace from definition \ref{V_A[k,l]} is replaced by the tensor product in which this particular subspace can be embedded. For example, putting $(k,l) = (4,2)$, this would lead to the following graphical representation: 
\begin{equation*}
\mV[4,2]\ \sim\ 
\begin{array}{ccccccc}
(4)\otimes(2) & \rightarrow & (5)\otimes(3) & \rightarrow & (6)\otimes(4) & \rightarrow & \cdots\\
\downarrow & & \downarrow & & \downarrow & & \\
(3)\otimes(3) & \rightarrow & (4)\otimes(4) & \rightarrow & (5)\otimes(5) & \rightarrow & \cdots\\
\downarrow & & \downarrow & & \downarrow & & \\
(2)\otimes(4) & \rightarrow & (3)\otimes(5) & \rightarrow & (4)\otimes(6) & \rightarrow & \cdots
\end{array}
\end{equation*}
In this picture, the horizontal and vertical arrows represent the $C$-action and the $S_u$-action respectively. What remains to be investigated is the action of the operators $A$ and $S_x$. As we will prove, these operators correspond to the horizontal and vertical  action `in the opposite direction', which is what one should intuitively expect. The action of the operators $A$ and $S_x$ will also define our projection operators (see section \ref{section 4.3})
\begin{equation*}
\Pi_{i,j} : \mcP_{p,q}(\mR^{2m},\mC) \cap \ker\big(\Delta_x,\Delta_u\big) \longrightarrow \mcH_{p-i+j,q-i-j}(\mR^{2m},\mC)\ . 
\end{equation*}
We will first investigate the action of the operator $S_x$ on the module $\mV[k,l]$ as this is the easiest one: 
\begin{proposition}\label{Sx_action}
The action of the operator $S_x$ on $\mV[k,l]$ is given by: 
\begin{equation*} 
S_x : \mV[k,l] \rightarrow \mV[k,l] : C^iS_u^j\mcH_{k,l}\ \mapsto\ \varphi_{i,j}(k,l) C^{i}S_u^{j-1}\mcH_{k,l}\ , 
\end{equation*}
where we have introduced the factor
\begin{equation*} 
\varphi_{i,j}(k,l) = j\left(k - l - (j-1)\right)\frac{l + j + \frac{m}{2} - 2}{l + i + j + \frac{m}{2} - 2}\ .  
\end{equation*}
\end{proposition}
\begin{proof}
First of all, repeated application of relation (\ref{C_Sx}) leads to
\begin{equation*} S_xC^iS_u^j\mcH_{k,l} = \left(\frac{H_u + 2}{H_u + 1}C\right)^iS_xS_u^j\mcH_{k,l} = \frac{H_u + i + 1}{H_u + 1}C^iS_xS_u^j\mcH_{k,l}\ . 
\end{equation*}
Using relations (\ref{A_Su}) and (\ref{Sx_Su}), it is clear that 
\begin{align*}
S_xS_u^j\bigg\vert_{\mcH_{k,l}} & =  \left(S_uS_x + \frac{H_x - H_u}{(1+H_x)(1+H_u)}CA - (H_x - H_u)\right)S_u^{j-1}\bigg\vert_{\mcH_{k,l}}\\
& =  S_uS_xS_u^{j-1} - (H_x - H_u)S_u^{j-1}\bigg\vert_{\mcH_{k,l}}\ ,
\end{align*}
which can be reduced by an induction argument to 
\begin{equation*} 
S_xS_u^j\bigg\vert_{\mcH_{k,l}} =  -j(H_x - H_u - j+1)S_u^{j-1}\bigg\vert_{\mcH_{k,l}}\ . 
\end{equation*}
Since $[H_x - H_u,C] = 0$, we are indeed lead to the formula above.
\end{proof}
\noindent
More generally, we can prove the following:
\begin{lemma}\label{S_xCS_u}
For all $i, j \in \mN$, we have the following: 
\begin{equation*}
 S_x^jC^iS_u^j\mcH_{k,l} = j!(k - l-j+1)^{(j)}\frac{\left(l + \frac{m}{2} - 1\right)^{(j)}}{\left(l + i + \frac{m}{2} - 1\right)^{(j)}}C^i\mcH_{k,l}\ , 
\end{equation*}
where we have introduced the upper factorial (for $\alpha \in \mR$ and $j \in \mN$)
\begin{equation*}
 \alpha^{(j)} = \alpha(\alpha+1)\ldots(\alpha + j - 1)\ .
\end{equation*}
\end{lemma}
\begin{proof}
Invoking proposition \ref{Sx_action}, we easily find that
\begin{align*}
S_x^jC^iS_u^j\mcH_{k,l} & =  \brac{\prod_{p = 1}^j\varphi_{i,p}(k,l)}C^i\mcH_{k,l}\\
& =  \brac{\prod_{p=1}^jp\big(k - l - (p-1)\big)\frac{\left(l + \frac{m}{2} + p - 2\right)}{\left(l + i + \frac{m}{2} + p - 2\right)}}C^i\mcH_{k,l}\ ,
\end{align*}
which can easily be rewritten as the formula above. 
\end{proof}
\noindent
Because of Lemmas \ref{lemmaSx_Su} and \ref{AC}, the action of the operator $A$ on $C^j\mcH_{k,l}$  is proportional to $C^{j-1}\mcH_{k,l}$. To investigate this action, we make the following definition:
\begin{definition}\label{Def-A-action}
For all $i \in \mN$, we define the constant $c_i(k,l)$ by means of
\begin{equation*} 
A : \mV[k,l] \rightarrow \mV[k,l] : C^i\mcH_{k,l} \mapsto c_i(k,l)C^{i-1}\mcH_{k,l}\ . 
\end{equation*}
\end{definition}
\noindent 
We want to study this constant in more detail, so let us introduce a few notations. For that purpose, we first rewrite the operator $AC$, taking into account that it is meant to act on $C^i\mcH_{k,l}$. From earlier calculations, using the fact that $S_x$ acts trivially on $C^i\mcH_{k,l}$ for all integers $i \in \mN$, we have that 
\begin{align*}
AC\bigg\vert_{C^i\mcH_{k,l}} & =  \left(1-\frac{1}{(1+H_x)(1 + H_u)}\right)CA - (H_x + H_u) + \frac{[S_x,S_u]}{1+H_x}\\
& =  \psi_1(H_x,H_u)CA + \psi_2(H_x,H_u)\ ,
\end{align*}
where we have introduced the operators
\begin{align*}
\psi_1(H_x,H_u)  :&=  1+\left( \frac{H_x - H_u}{(1+H_x)^2(1+H_u)}-\frac{1}{(1+H_x)(1 + H_u)}\right)=\frac{H_x(H_x+2)}{(H_x+1)^2} \\ 
\intertext{and} 
\psi_2(H_x,H_u) :&=  \frac{H_u - H_x}{1 + H_x} - (H_u + H_x)\ =\ -\frac{H_x(H_x + H_u + 2)}{1 + H_x}\ .
\end{align*}
The eigenvalues of these operators on homogeneous polynomials in $(x,u)$ will get their own symbol, where from now on we use the shorthand notation 
\begin{equation*} 
(K,L) := \left(k + \frac{m}{2},l + \frac{m}{2}\right)\ . 
\end{equation*}
\begin{definition}
On homogeneous polynomials $P_{k,l}(x,u)$ of degree $(k,l)$ on $\mR^{2m}$ we introduce the following eigenvalues for the operators $\psi_j(H_x,H_u)$ from above: 
\begin{align*}
\psi_1(k,l) & :=  \frac{K(K-2)}{(K-1)^2} \\
\psi_2(k,l) & :=  \frac{K(K + L - 2)}{K - 1}
\end{align*}
\end{definition}
\noindent 
Note that $\psi_1(k,l),\psi_2(k,l) > 0$ for all $P_{k,l}(x,u)$.
\begin{lemma}\label{recurrence}
The constant $c_i(k,l)$ satisfies the following recursive relation: 
\begin{equation*} 
c_{i+1}(k,l) = \psi_1(k+i,l+i)c_i(k,l) + \psi_2(k+i,l+i)\ . 
\end{equation*}
\end{lemma}
\begin{proof}
This follows from the fact that 
\begin{align*}
AC^{i+1}\mcH_{k,l} & =  \big(\psi_1(H_x,H_u)CA + \psi_2(H_x,H_u)\big)C^i\mcH_{k,l}\\
& =  \big(\psi_1(k+i,l+i)c_i(k,l) + \psi_2(k+i,l+i)\big)C^i\mcH_{k,l}\ ,
\end{align*}
as elements in $C^i\mcH_{k,l}$ are homogeneous of degree $(k+i,l+i)$ in the variables $(x,u)$. 
\end{proof}
\noindent
Note that $c_i(k,l) \neq 0$ for all integers $(k,l)$ and $i > 0$ (see below),  so we have the following conclusion:
\begin{equation*}
A^iC^i\mcH_{k,l} = \left(\prod_{s = 1}^ic_s(k,l)\right)\mcH_{k,l} \neq 0\ , 
\end{equation*}
which after inversion yields 
\begin{equation*}
A^iC^i\left(\frac{1}{\prod_{s = 1}^ic_s(k,l)} H_{k,l}(x,u)\right) = H_{k,l}(x,u)\ .
\end{equation*}
The following proposition explicitly solves the recursive relation from lemma \ref{recurrence}:
\begin{proposition}
For all integers $i>0$, the constant $c_i(k,l)$ is of the following form:
\begin{equation*}
c_i(k,l)=i\frac{\brac{k+\frac{m}{2}+i-1}(k+l+m+i-3)}{k+\frac{m}{2}+i-2}
\end{equation*} 
\end{proposition}
\begin{proof}
To prove the proposition, we will make use of the recurrence relations obtained in lemma \ref{recurrence}:
\begin{equation*} 
c_{i+1}(k,l) = \psi_1(k+i,l+i)c_i(k,l) + \psi_2(k+i,l+i)\ . 
\end{equation*}
Substituting the expressions for $\psi_1(k+i,l+i),\psi_2(k+i,l+i)$, we obtain
\begin{equation}\label{ProofRecurrence}
c_{i+1}(k,l)=\frac{K+i}{K+i-1}\brac{\frac{K+i-2}{K+i-1}c_i(k,l)+(K+L+2(i-1)},
\end{equation}
where we used the shorthand notation $(K,L) := \left(k + \frac{m}{2},l + \frac{m}{2}\right)$ introduced earlier. Since it is known from definition \ref{Def-A-action} that $c_0(k,l)=0$, we have
\begin{equation*}
c_1(k,l)=\psi_2(k,l)=\frac{K(K + L - 2)}{K - 1}.
\end{equation*}
Should the expression for $c_i(k,l)$ be correct for all integers $i>0$, then we have to obtain the expression for $c_{i+1}(k,l)$ after substituting $c_i(k,l)$ in equation (\ref{ProofRecurrence}). This yields: 
\begin{align*}
c_{i+1}(k,l)&=\frac{K+i}{K+i-1}\brac{j(K+L+i-3)+(K+L+2(i-1)} \\
&=(i+1)\frac{(K+i)(K+L+i-2)}{K+i-1},
\end{align*}
which is indeed of the correct form. 
\end{proof}
\noindent
Finally, we have the following: 
\begin{proposition}\label{A-action}
The action of the operator $A$ is given by: 
\begin{equation*}
 A : \mV[k,l] \rightarrow \mV[k,l] : C^iS_u^j\mcH_{k,l} \mapsto \psi_{i,j}(k,l)C^{i-1}S_u^{j}\mcH_{k,l}\ , 
\end{equation*}
where we have introduced the factor
\begin{equation*}
\psi_{i,j}(k,l) =i \frac{\left(k + \frac{m}{2} + i - 1\right)\left(l + \frac{m}{2} + i - 2\right)\brac{k+l+m+i-3}}{\left(k + \frac{m}{2} + i -j-2\right)\left(l + \frac{m}{2} + i +j- 2\right)}\ .
\end{equation*}
\end{proposition}
\begin{proof}
First of all, one can easily prove by induction on the exponent $i \in \mN$ that the operator $A \in Z\big(\g,\gs\big)$ indeed maps $C^iS_u^j\mcH_{k,l} \mapsto C^{i-1}S_u^{j}\mcH_{k,l}$. It suffices to use lemma \ref{AC} to see this. Next, we need to find an expression for the constant $\psi_{i,j}(k,l)$. To do so, we calculate
\begin{align*}
S_u^jS_x^jAC^iS_u^j\mcH_{k,l} & =  \psi_{i,j}(k,l)S_u^jS_x^jC^{i-1}S_u^j\mcH_{k,l}\\
& =  \psi_{i,j}(k,l)\left(j!(k - l-j+1)^{(j)}\frac{\left(l + \frac{m}{2} - 1\right)^{(j)}}{\left(l + i + \frac{m}{2} - 2\right)^{(j)}}\right)S_u^jC^{i-1}\mcH_{k,l}\ .
\end{align*}
On the other hand, we can also make use of the fact that 
\begin{equation*}
S_x^jA = \frac{H_x + 1}{H_x + j + 1}AS_x^j
\end{equation*}
to calculate the same expression ([LHS], for left-hand side) in a second way: 
\begin{align*}
\mbox{[LHS]} & =  \frac{k + i + \frac{m}{2} - 2}{k + i + \frac{m}{2} - 2 - j}S_u^jAS_x^jC^iS_u^j\mcH_{k,l}\\
& =  \frac{k + i + \frac{m}{2} - 2}{k + i + \frac{m}{2} - 2 - j}\left(j!(k - l-j+1)^{(j)}\frac{\left(l + \frac{m}{2} - 1\right)^{(j)}}{\left(l + i + \frac{m}{2} - 1\right)^{(j)}}\right)c_i(k,l)S_u^jC^{i-1}\mcH_{k,l}
\end{align*}
from which the result then indeed follows. 
\end{proof}
\noindent
More generally, we can prove the following:
\begin{proposition}\label{AS_x-action}
For $i \geq p$ and $j \geq q$, the action of the operator $A^pS_x^q$ is given by: 
\begin{equation*}
A^pS_x^q : \mV[k,l] \rightarrow \mV[k,l] : C^iS_u^j\mcH_{k,l} \mapsto \alpha_{i,j}^{p,q}(k,l)C^{i-p}S_u^{j-q}\mcH_{k,l}\ , 
\end{equation*}
where the constant $\alpha_{i,j}^{p,q}(k,l)$ is defined as:
\begin{equation*}
\begin{split}
\alpha_{i,j}^{p,q}(k,l):=&(i)_{(p)}(j)_{(q)}(k - l-j+1)^{(q)}\frac{\brac{k+\frac{m}{2}+i-1}_{(p)}}{\brac{k+\frac{m}{2}+i-j+q-2}_{(p)} } \\ 
&\times \frac{ \left(l + \frac{m}{2}+j-2\right)_{(q)}\brac{l+\frac{m}{2}+i-2}_{(p)} }{ \left(l+\frac{m}{2}+i+j-2\right)_{(q)} \brac{l+\frac{m}{2}+i+j-q-2}_{(p)} }\brac{k+l+m+i-3}_{(p)},
\end{split}
\end{equation*}
where we have introduced the falling factorial (for $\alpha \in \mR$ and $j \in \mN$)
\begin{equation*}
 \alpha_{(j)} = \alpha(\alpha-1)\ldots(\alpha - j + 1)\ .
\end{equation*}
\end{proposition}
\begin{proof}
The proof follows from a slightly modified version of lemma \ref{S_xCS_u} and repeated application of definition \ref{Def-A-action} and proposition \ref{A-action}. 
\end{proof}
\noindent


\subsection{Main result of the paper} 

We will now prove that the space $\mcP(\mR^{2m},\mC)\cap \ker\big(\Delta_x,\Delta_u\big)$ can be decomposed into irreducible 
representation for $Z\big(\g,\gs\big) \times \mathrm{SO}(m)$ and that the decomposition is multiplicity free. To do so, we need the infinite-dimensional module $\mV[k,l]$ for the algebra $Z\big(\g,\gs\big)$, defined for all pairs 
of integers $k \geq l \geq 0$. 
\begin{theorem}\label{Main-theorem}
(Generalised Howe duality) Under the joint action of $Z\big(\g,\gs\big) \times \mathrm{SO}(m)$, the space $\mcP(\mR^{2m},\mC)\cap \ker\big(\Delta_x,\Delta_u\big) $ decomposes as follows:
\begin{equation*}
\mcP(\mR^{2m},\mC)\cap \ker\big(\Delta_x,\Delta_u\big)  \cong \bigoplus_{k \geq l} \mV\comm{k,l} \otimes \mcH_{k,l}(\mR^{2m},\mC)\ . 
\end{equation*}
The summation is hereby performed over all pairs of integers $(k,l) \in \mZ^+\times\mZ^+$ satisfying the dominant weight condition $k \geq l \geq 0$.
\end{theorem}
\begin{proof}
To prove the isomorphism, we start with the following: there is an injective $Z\big(\g,\gs\big) \times \mathrm{SO}(m)$-intertwining map given by
\begin{equation*}
\begin{split}
\mV\comm{k,l} \otimes \mcH_{k,l}(\mR^{2m},\mC) &\longrightarrow \mcP(\mR^{2m},\mC)\cap \ker\big(\Delta_x,\Delta_u\big) \\
\Phi(C)S_u^j\, 1_{\comm{k,l}}\otimes H_{k,l}(x,u) &\mapsto \Phi(C)S_u^jH_{k,l}(x,u),
\end{split}
\end{equation*}
where $\Phi(C)$ is a polynomial in the variable $C$ and $1_{\comm{k,l}}$ is the highest weight vector of the $Z\big(\g,\gs\big)$-module $\mV\comm{k,l}$ introduced in the previous section with the action of the generators as in propositions \ref{(C,S_u)-action}, \ref{Sx_action} and \ref{A-action}. We still need to prove that the intertwining map is surjective. To do so, it suffices to prove that we can obtain $\mcP_{k,l}\brac{\mR^{2m},\mC}\cap \ker(\Delta_x,\Delta_u)$, for every $(k,l)\in \mN \times \mN$. We will assume that $k \geq l$, as the proof for the other case is completely similar. Using theorem \ref{Klymik} and counting of the degrees of homogeneity, we obtain
\begin{equation*}
\mcP_{k,l}\brac{\mR^{2m},\mC}\cap \ker(\Delta_x,\Delta_u)=\bigoplus_{i=0}^{l}\bigoplus_{j=0}^{l-i} C^iS_u^j\mcH_{k-i+j,l-i-j}\brac{\mR^{2m},\mC} \ ,
\end{equation*}
which completes the proof of the theorem. 
\end{proof}
\noindent
As a consequence of theorem \ref{Main-theorem}, we have the following result:
\begin{corollary}\label{Main-decomposition}
From the $Z\big(\g,\gs\big)\:$-module $\mV[k,l]$, we obtain an explicit decomposition of the space of double harmonics into simplicial harmonics: 
\begin{equation*}
\mcP_{k,l}\brac{\mR^{2m},\mC}\cap \ker(\Delta_x,\Delta_u)=\bigoplus_{i=0}^{l}\bigoplus_{j=0}^{l-i} C^iS_u^j\mcH_{k-i+j,l-i-j}\brac{\mR^{2m},\mC}.
\end{equation*}
\end{corollary}


\subsection{Projection operators}
\label{section 4.3}
To conclude this section, we will construct the projection operators $\Pi_{i,j}$ from \mbox{section 2}. We first start with a projection onto the simplicial harmonic part of a polynomial in $\ker\big(\Delta_x,\Delta_u\big)$:
\begin{theorem}\label{main-projection}
The projection operator 
\begin{equation*}
\Pi: \mcP_{k,l}(\mR^{2m},\mC) \cap \ker\big(\Delta_x,\Delta_u\big) \longrightarrow \mcH_{k,l}(\mR^{2m},\mC)\ , 
\end{equation*}
is defined as
\begin{equation*}
\Pi = \sum_{i,j=0}^{+\infty}\beta_{i,j}(k,l)C^iS_u^jA^iS_x^j,
\end{equation*}
where the coefficients $\beta_{i,j}(k,l)$ are given by: 
\begin{equation*}
\beta_{i,j}(k,l)=\frac{(-1)^{i+j}}{i!j!}\frac{(k+\frac{m}{2}-i+j-1)(l+\frac{m}{2}-j-3)_{(i)}}{(k+\frac{m}{2}+j-1)(l+\frac{m}{2}-3)_{(i)}(k+l+m-4)_{(i)}(k-l+2)^{(j)}}\ .
\end{equation*}
\end{theorem}
\begin{proof}
It is clear that $\Pi\comm{\mcH_{k,l}}=\mcH_{k,l}$, so it remains to prove that  
\begin{equation*}
\Pi\comm{C^iS_u^j\mcH_{k-i+j,l-i-j}}=0,
\end{equation*}
which will be done in the appendix. 
\end{proof}
\noindent
Using this theorem, we can now finally define the projection operators $\Pi_{i,j}$ in terms of the generators $A$ and $S_x$ of the transvector algebra $Z\big(\g,\gs\big)$:
\begin{theorem}\label{Projection_operator_2}
For all $0\leq i \leq q$ and $0\leq j \leq q-i$, the projection operator
\begin{equation*}
\Pi_{i,j} : \mcP_{p,q}(\mR^{2m},\mC) \cap \ker\big(\Delta_x,\Delta_u\big) \longrightarrow \mcH_{p-i+j,q-i-j}(\mR^{2m},\mC)\ , 
\end{equation*}
is given by
\begin{equation*}
\Pi_{i,j}\comm{P_{p,q}(x,u)}:=\frac{1}{\alpha_{i,j}^{i,j}(k,l)}\Pi A^iS_x^j\comm{P_{p,q}(x,u)}\ ,
\end{equation*}
where the constant $\alpha_{i,j}^{i,j}(k,l)$ was given in proposition \ref{AS_x-action}.
\end{theorem}
\begin{proof}
It suffices to note that the action of the operator $\frac{A^iS_x^j}{\alpha_{i,j}^{i,j}(k,l)}$ maps a polynomial in $\mcP_{p,q}(\mR^{2m},\mC) \cap\ker\big(\Delta_x,\Delta_u\big)$ to a polynomial in $\mcP_{p-i+j,q-i-j}(\mR^{2m},\mC) \cap \ker\big(\Delta_x,\Delta_u\big)$. The result then follows from theorem \ref{main-projection}.
\end{proof}


\section{An explicit example} \label{example}

Let us consider an explicit example of how to do the decomposition of a polynomial in two vector variables into polynomials belonging to irreducible representations for the H-action of the orthogonal group. The main idea is the following: 
\begin{itemize}
\item First we will use the powers in $|x|^2$ and $|u|^2$ to reduce the problem to a tensor product of harmonics. 
\item For a space of the form $\mcP_{p,q}(\mR^{2m},\mC) \cap \ker(\Delta_x,\Delta_u)$ one can define an ordering on the summands, which allows an inductive procedure. Indeed, when looking at the embedding factors we will consider the lexicographic ordering $(a,b)$ with respect to the powers in $S_u$ and $C$. In other words, an embedding factor of the form $S_u^aC^b$ is described by a label of the form $(a,b)$.
\end{itemize}
\noindent
Let us consider a polynomial $P_{3,2}(x,u)\in \mcP_{3,2}(\mR^{2m},\mC)$. First of all, it is clear that 
\begin{align*}
P_{3,2}(x,u) & =  H_{3,2}'(x,u) + |u|^2H_{3,0}'(x,u) + |x|^2H_{1,2}'(x,u) + |x|^2|u|^2H_{1,0}'(x)\ .
\end{align*}
Recall that the prime-superscript was used for polynomials in $\ker(\Delta_x,\Delta_u)$. Projecting the given polynomial on each of these components is fairly easy, as it corresponds to applying the projection operators from theorem \ref{Projection_operator_1} (which is merely a product of projection operators from the one variable case, respectively in $x$ and $u$). \\
\\
\noindent
The upshot is of course that once this projection has been carried out, we still need to take theorem \ref{Main-theorem} and corollary \ref{Main-decomposition} into account. Indeed, one for example has that 
\begin{equation*}
\begin{split}
H'_{3,2}(x,u) &= S_u^2 H_{5}(x) + CS_u H_{3}(x) + S_u H_{4,1}(x,u) \\
&+ C^2 H_{1}(x) + C H_{2,1}(x,u)+ H_{3,2}(x,u)\ .
\end{split}
\end{equation*}
The projection on each of the components above then requires the knowledge of the constants $\alpha_{i,j}^{p,q}$ from proposition \ref{AS_x-action}. Using the lexicographic ordering on the components mentioned at the start of this section, we first determine the component $H_5(x)\in\mcH_{5}$ of $H'_{3,2}(x,u)$: 
\begin{equation*}
H_5(x)=\frac{1}{40}S_x^2H'_{3,2}(x,u)\ .
\end{equation*} 
We can proceed by subtracting $S_u^2H_5(x)$ from $H'_{3,2}(x,u)$ and applying the operator $AS_x$ to find the component $H_3(x)\in \mcH_{3}$ of $H'_{3,2}(x,u)$. In terms of the ordering, we note that there are two summands having a first power in $S_u$ in the embedding, but it suffices to look at the action of the operator $A$ as this singles out a unique component: 
\begin{align*}
H_3(x)&=\frac{m(m+4)}{3(m-2)(m+1)(m+6)}AS_x\brac{H'_{3,2}(x,u)-S_u^2H_5(x)} \\
&=\frac{m(m+4)}{3(m-2)(m+1)(m+6)}AS_xH'_{3,2}(x,u)\ .
\end{align*}
In the last equality, the term $S_u^2H_5(x)$ was omitted because $AS_u^2H_5(x)=0$. Next we determine $H_{4,1}(x,u)\in \mcH_{4,1}$, which can be obtained by subtracting $S_u^2H_5(x)+CS_uH_3(x)$ from $H'_{3,2}(x,u)$ and applying the operator $S_x$:
\begin{equation*}
H_{4,1}(x,u)=\frac{1}{3}S_x\brac{H'_{3,2}(x,u)-S_u^2H_5(x)-CS_uH_3(x)} \ .
\end{equation*}
Finally, there are three terms having no powers of $S_u$ in their embedding factors but these can again be separated in terms of the power of $C$. So, first of all the component $H_1(x)\in \mcH_{1}$ of $H'_{3,2}(x,u)$ is obtained as follows (the subtracted terms are omitted because the action of $A^2$ on these terms is trivial):
\begin{equation*}
H_1(x)=\frac{1}{2(m-1)(m+4)} A^2H'_{3,2}(x,u)\ .
\end{equation*}
The component $H_{2,1}(x,u)$ can be obtained as follows:
\begin{equation*}
H_{2,1}(x,u)=\frac{m+2}{(m+1)(m+4)}A\brac{H'_{3,2}(x,u)-S_u^2H_5(x)-CS_uH_3(x)-C^2H_1(x)} \ .
\end{equation*}
Finally, $H_{3,2}(x,u)$ is obtained by subtracting all other components with the proper embeddings from $H'_{3,2}(x,u)$.


\section{Fischer inner product and orthogonality}\label{orthogonality}

We will explore the orthogonality of the components in the Fischer decomposition w.r.t. the Fischer inner product on $\mcP\brac{\mR^{2m},\mC}$. 
\begin{definition}
The positive definite Fischer inner product on $\mcP(\mR^{2m},\mC)$ is defined as
\begin{equation}\label{Fischer}
\comm{P(x,u) , Q(x,u)}_F := \overline{P(\partial_x,\partial_u)}Q(x,u)\vert_{x = u = 0}\ .
\end{equation} 
The operators $\D_{x}$ and $\D_{u}$ on the right hand side of the expression denote that each variable in $P(x,u)$ is replaced by its corresponding partial derivative. 
\end{definition}
\noindent 
We will first have a look at the Fischer adjoints of the generators of $Z(\g,\gs)$. To do so, we need some lemmas:
\begin{lemma}
The extremal projection operator $\pi_{\gs}$ is self-adjoint with respect to the \mbox{Fischer} inner product, i.e. 
\begin{equation*}
\pi_{\gs}^{\dagger}=\pi_{\gs}.
\end{equation*}
\end{lemma}
\begin{proof}
It suffices to proof that $\pi_x$ is self-adjoint, as the proof for $\pi_u$ is completely similar. We have the following:
\begin{align*}
\pi_x^{\dagger}&=\brac{\sum_{j=0}^\infty\frac{1}{4^jj!}\frac{\Gamma(H_x + 2)}{\Gamma(H_x + 2 + j)}\nx^{2j}\Delta_x^j}^{\dagger} \\
&= \sum_{j=0}^\infty\frac{1}{4^jj!}\brac{\frac{\Gamma(H_x + 2)}{\Gamma(H_x + 2 + j)}\nx^{2j}\Delta_x^j}^{\dagger}  \\
&= \sum_{j=0}^\infty\frac{1}{4^jj!}\frac{\Gamma(H_x + 2)}{\Gamma(H_x + 2 + j)}\brac{\Delta_x^j}^{\dagger} \brac{\nx^{2j}}^{\dagger}.
\end{align*}
The last equality follows from the fact that $H_x$ is self-adjoint and the fact that $\nx^{2j}\Delta_x^j$ is $0$-homogeneous. Since $\Delta_x$ and $ \nx^{2}$ are Fischer adjoint, it follows that $\pi_x^{\dagger}=\pi_x$. The proof follows from the fact that $\pi_{\gs}^{\dagger}=\brac{\pi_x\pi_u}^{\dagger}=\pi_u\pi_x=\pi_{\gs}$.
\end{proof}
\begin{lemma}
The Fischer adjoints of the operators in $\gt$, which are operators acting on $\mcP(\mR^{2m},\mC)$, are given by:
\begin{alignat*}{3}
\inner{u,x}^{\dagger}&=\inner{\D_u,\D_x} &\quad &\text{and} &\quad \inner{\D_u,\D_x}^{\dagger}&=\inner{u,x} \\
\inner{u,\D_x}^{\dagger}&=\inner{x,\D_u} &\quad &\text{and} &\quad \inner{x,\D_u}^{\dagger}&=\inner{u,\D_x} \\
\end{alignat*} 
\end{lemma}
\vspace{-1cm}
\begin{proof} 
The result immediately follows from the fact that the operators $x_j$ and $\D_{x_j}$, and $u_j$ and $\D_{u_j}$ with $1\leq j \leq m$, are each others Fischer adjoints. 
\end{proof}
\noindent
Combining the two previous lemmas, we arrive at the following result: 
\begin{proposition}
The adjoints of the generators of $Z(\g,\gs)$, which are well-defined on $\mcP(\mR^{2m},\mC)\cap \ker{\brac{\Delta_x,\Delta_u}}$, are given by:
\begin{alignat*}{3}
C^{\dagger}&=A &\quad &\text{and}&\quad A^{\dagger}&=C \\
S_u^{\dagger}&=S_x^{\phantom{\dagger}} &\quad &\text{and}&\quad S_x^{\dagger}&=S_u^{\phantom{\dagger}}
\end{alignat*}
\end{proposition}
\begin{proof} 
Since $\pi_{\gs}P=P$ and $\pi_{\gs}Q=Q$ for $P,Q\in\mcP(\mR^{2m},\mC)\cap \ker{\brac{\Delta_x,\Delta_u}}$, we have:
\begin{align*}
\comm{\pi_{\gs}\inner{u,x}P,Q}_F&=\comm{\inner{u,x}P,\pi_{\gs}Q}_F \\
&=\comm{\pi_{\gs}P,\inner{\D_u,\D_x}Q(x,u)}_F \\
&=\comm{P,\pi_{\gs}\inner{\D_u,\D_x}Q(x,u)}_F.
\end{align*}
The other equalities are proven in a similar way. 
\end{proof}
\noindent
We are now in a position to prove that the projection operator from theorem \ref{main-projection} is self-adjoint: 
\begin{theorem}
The projection operator 
\begin{equation*}
\Pi: \mcP_{k,l}(\mR^{2m},\mC) \cap \ker\big(\Delta_x,\Delta_u\big) \longrightarrow \mcH_{k,l}(\mR^{2m},\mC)\ , 
\end{equation*}
defined as
\begin{equation*}
\Pi = \sum_{i,j=0}^{+\infty}\beta_{i,j}(k,l)C^iS_u^jA^iS_x^j
\end{equation*}
is self-adjoint with respect to the Fischer inner product.
\end{theorem}
\begin{proof}
We will prove that expressions of the form $C^iS_u^jA^iS_x^j$ are self-adjoint. We have
\begin{equation*}
\brac{C^iS_u^jA^iS_x^j}^{\dagger}=S_u^jC^iS_x^jA^i.
\end{equation*}
We then need to swap $S_u^j$ with $C^i$ and $S_x^j$ with $A^i$. Similarly to what was found in the proof of proposition \ref{Sx_action}, we have
\begin{equation}
S_uC^i=\frac{H_x+i+1}{H_x+1}C^iS_u.
\label{relation}
\end{equation}
We then can proceed inductively:
\begin{align*}
S_u^jC^i&=S_u^{j-1}\frac{H_x+i+1}{H_x+1}C^iS_u \\
&=S_u^{j-2}\frac{H_x+i}{H_x}S_uC^iS_u \\
&=S_u^{j-2}\frac{(H_x+i)(H_x+i+1)}{H_x(H_x+1)}C^iS_u^2 \\
&\; \vdots \\
&=\frac{(H_x+i+1)_{(j)}}{(H_x+1)_{(j)}}C^iS_u^j.
\end{align*}
Using lemma \ref{scaled_commutation}, we prove a relation similar to (\ref{relation}):
\begin{equation*}
S_xA^i=\frac{H_x-i+2}{H_x+2}A^i S_x.
\end{equation*}
This can be used to swap $S_x^j$ with $A^i$:
\begin{align*}
S_x^jA^i&=S_x^{j-1}\frac{H_x-i+2}{H_x+2}A^iS_x \\
&=S_u^{j-2}\frac{H_x-i+3}{H_x+3}S_xA^iS_x \\
&\; \vdots \\
&=\frac{(H_x-i+2)^{(j)}}{(H_x+2)^{(j)}}A^iS_x^j.
\end{align*}
Putting everything together, we find: 
\begin{align*}
\brac{C^iS_u^jA^iS_x^j}^{\dagger}&=\frac{(H_x+i+1)_{(j)}}{(H_x+1)_{(j)}}C^iS_u^j\frac{(H_x-i+2)^{(j)}}{(H_x+2)^{(j)}}A^iS_x^j\\
&=C^iS_u^jA^iS_x^j,
\end{align*}
which proves the statement.
\end{proof}
\noindent
Finally, we will prove that the irreducible pieces in the decomposition of the space of double harmonics $\mcP_{k,l}(\mR^{2m},\mC)\cap \ker{(\Delta_x,\Delta_u)}$ are all orthogonal with respect to the Fischer inner product on $\mcP_{k,l}(\mR^{2m},\mC)$. 
\begin{theorem}
For all $\alpha_1,\alpha_2\in \set{0,\ldots,l}$ and $\beta_n\in \set{0,\ldots,l-\alpha_n}$, with $n=1,2$, the following holds: if $\alpha_1\neq\alpha_2$ and $\beta_1\neq\beta_2$, then
\begin{equation*}
\comm{C^{\alpha_1}S_u^{\beta_1}H(x,u),C^{\alpha_2}S_u^{\beta_2}G(x,u)}_F=0, 
\end{equation*}
for all $H(x,u)\in\mcH_{k-\alpha_1+\beta_1,l-\alpha_1-\beta_1}(\mR^{2m},\mC)$ and $G(x,u)\in\mcH_{k-\alpha_2+\beta_2,l-\alpha_2-\beta_2}(\mR^{2m},\mC)$.
\end{theorem}
\begin{proof}
Since the Fischer inner product is symmetric, it is not a restriction to assume that $\alpha_1\geq\alpha_2$ and $\beta_1\geq\beta_2$. Take $H(x,u)$ and $G(x,u)$ as stated in the theorem. We then have
\begin{equation*}
\comm{C^{\alpha_1}S_u^{\beta_1}H(x,u),C^{\alpha_2}S_u^{\beta_2}G(x,u)}_F = \comm{S_u^{\beta_1}H(x,u),A^{\alpha_1}C^{\alpha_2}S_u^{\beta_2}G(x,u)}_F.
\end{equation*}
From repeated application of lemma \ref{scaled_commutation} and theorem \ref{A-action}, it follows that this expression is zero unless $\alpha_1=\alpha_2$. We also have
\begin{align*}
\comm{C^{\alpha_1}S_u^{\beta_1}H(x,u),C^{\alpha_2}S_u^{\beta_2}G(x,u)}_F &=c(\alpha_1,\beta_1)\comm{S_u^{\beta_1}C^{\alpha_1}H(x,u),C^{\alpha_2}S_u^{\beta_2}G(x,u)}_F \\
&=c(\alpha_1,\beta_1)\comm{C^{\alpha_1}H(x,u),S_x^{\beta_1}C^{\alpha_2}S_u^{\beta_2}G(x,u)}_F,
\end{align*}
where the constant $c(\alpha_1,\beta_1)$ is given by
\begin{equation*}
c(\alpha_1,\beta_1)=\frac{(H_x+1)_{(\beta_1)}}{(H_x+\alpha_1+1)_{(\beta_1)}}.
\end{equation*}
It follows from proposition \ref{Sx_action}, that the inner product is zero unless $\beta_1=\beta_2$. 
\end{proof}
\noindent
Putting this result together with the classical orthogonality of harmonics in the Fischer decomposition we get:
\begin{corollary}
The full decomposition of $\mcP_{k,l}(\mR^{2m},\mC)$ into irreducible representations of the orthogonal group is a decomposition into orthogonal subspaces with respect to the Fischer inner product. 
\end{corollary}


\section{Pizzetti's formula on the Stiefel manifold $V_2(\mR^m)$}\label{Pizzetti}

In this final section, we will obtain a generalisation of Pizzetti's classical formula (for polynomials in one vector variable) to polynomials in two variables. The original proof of the classical result is analytic but the proof in \cite{HDB2} heavily exploits the Lie algebra (\ref{harm_sl}) underlying analysis on $\mR^m$, and expresses the integral of a polynomial $P(\omega)$ on the unit sphere $S^{m-1} \subset \mR^m$ (i.e. the restriction of a polynomial $P(x)$ on $\mR^m$ for $|x| = 1$) as an $\mathrm{SO}(m)$-invariant functional 
\begin{equation*}
\mcI_1:\mcP(\mR^{2m},\mC)\longrightarrow \mC \ ,
\end{equation*}
which satisfies $\mcI_1(\norm{x}^2P(x))=\mcI(P(x))$. Schur's lemma then implies that $\mcI_1(\mcH_k)=0$, so that the only contribution to the integral is the term $\norm{x}^{2a}H_0$, with $H_0\in \mcH_0$. Using the Fischer decomposition (\ref{example_Fischer}), we can explicitly project on $\mcH_0$, which leads to a formal summation of the form
\begin{equation*} 
\int_{S^{m-1}}P(\omega)dS_1(\omega) = \sum_{k = 0}^\infty\frac{2\pi^{\frac{m}{2}}}{4^kk!\Gamma\left(k + \frac{m}{2}\right)}\Delta_x^kP(x)\bigg\vert_{x = 0}\ .
\end{equation*}
In order to generalise this result for harmonic analysis in two vector variables, we first introduce the analogue of the unit sphere $S^{m-1}$, which is the Stiefel manifold $V_2(\mR^m)$. This is a homogeneous space
\begin{equation*}
V_2(\mR^m) \cong \textup{SO}(m)/\textup{SO}(m-2)\ , 
\end{equation*}
which can be identified with the tangent bundle to the unit sphere in $\mR^m$, i.e. it contains elements $(\omega,\eta) \in S^{m-1} \times S^{m-1}$ for which $\langle\omega,\eta\rangle = 0$. This can also be written as $V_2(\mR^m) = \{A \in \mR^{m \times 2} : A^tA = 1 \in \mR^{2 \times 2}\}$. The integration of (restrictions of) polynomials in two vector variables over this particular manifold can then be defined as an $\mathrm{SO}(m)$-invariant functional 
\begin{equation*}
 \mcI_2 : \mcP(\mR^{2m},\mC) \rightarrow \mC : P(x,u) \mapsto \int_{V_2(\mR^m)}P(\omega,\eta)dS_2\ , 
\end{equation*}
where $dS_2$ is the measure, see e.g. \cite{CH} for an explicit definition. This measure is invariant under the left action of $\mathrm{SO}(m)$ and right action of $\mathrm{SO}(2)$. In other words, identifying the couple $(\omega,\eta)$ with matrices $A \in \mR^{m \times 2}$ satisfying the conditions from above, we have that $\mcI_2[P(gA)] = \mcI_2[P(A)] = \mcI_2[P(Ah)]$, for $g \in \mathrm{SO}(m)$ and $h \in \mathrm{SO}(2)$.  For all $P(x,u)\in \mcP(\mR^{2m},\mR)$, the integral satisfies moreover:
\begin{align*}
\mcI_2\brac{\norm{x}^2P(x,u)}&=\mcI_2\brac{\norm{u}^2P(x,u)}=\mcI_2\brac{P(x,u)} \\
\mcI_2\brac{\inner{u,x}P(x,u)}&=0\ .
\end{align*}
Using theorem \ref{Main-theorem}, we have for an arbitrary polynomial $P_{p,q}(x,u)\in \mcP_{p,q}(\mR,\mC)$ that
\begin{equation*}
P_{p,q}(x,u)=\sum_{a,b,\alpha,\beta}\norm{x}^{2a}\norm{u}^{2b}C^{\alpha}S_u^{\beta}H_{k,l}(x,u)\ ,
\end{equation*}
with $H_{k,l}(x,u)\in \mcH_{k,l}(\mR^{2m},\mC)$ and where the summation is such that 
\begin{equation*}
\begin{dcases}
p=2a-\alpha+\beta +k \\
q=2b+\alpha+\beta +l\ .
\end{dcases}
\end{equation*}
Since the integration is $\mathrm{SO}(m)$-invariant, we can again use Schur's lemma together with theorems \ref{Projection_operator_1} and \ref{Main-theorem} to state the following result:
\begin{equation*}
\mcI_2\brac{P_{p,q}(x,u)}=\begin{dcases}
   \mcI_2\brac{H_0}  &\quad \text{if } p=q+2n, \: \brac{n\in \mN} \\
    0    &\quad \text{otherwise}\ , \\
    \end{dcases}
\end{equation*}
where $P_{p,q}(x,u)$ is an arbitrary polynomial homogeneous of degree $(p,q)$ in $(x,u)$. In other words, we need a projection of $\mcP_{p,q}(\mR^{2m},\mC)$ onto $\mcH_0$, which will be done in two steps: first a projection on the space $\mcP_{p-2i,q-2j}(\mR^{2m},\mC) \cap \ker\big(\Delta_x,\Delta_u\big)$ of double harmonics, using theorem \ref{Projection_operator_1}. Already now, it is easy to see from theorem \ref{Klymik} that if $p-2i \neq q-2j$, no component $\mcH_0$ is present so we will assume that $p-2i = q-2j$. Note that we still need a projection on the component $\mcH_0$ but this is what will be established in what follows. \\ \\ \noindent
Before we continue, let us investigate how $\mcH_0$ is embedded in the space of double harmonics $\mcP_{\beta,\beta}(\mR^{2m},\mC) \cap \ker\big(\Delta_x,\Delta_u\big)$. From theorem \ref{Main-decomposition}, it is known that for $\beta \in \mN$
\begin{equation*}
C^{\beta}:\mcH_0 \longrightarrow \mcP_{\beta,\beta}(\mR^{2m},\mC) \cap \ker\big(\Delta_x,\Delta_u\big).
\end{equation*}
\begin{theorem}\label{Gegenbauer}
For $b:=\lfloor \frac{\beta}{2} \rfloor$, we have that $\norm{x}^{-2b}\norm{u}^{-2b}C^{\beta}\comm{1}$ is a polynomial in the variable $t:=\frac{\inner{u,x}}{\norm{x}\norm{u}}$ where $C^{\beta}\comm{1}$ belongs to $\ker\big(\Delta_x,\Delta_u\big)$. It is explicitly given by 
\begin{equation*}
C^{\beta}\comm{1}=\frac{\beta !\, \Gamma\brac{\frac{m}{2}-1}}{2^{\beta}\Gamma\brac{\beta+\frac{m}{2}-1}}\norm{x}^{2b}\norm{u}^{2b}C_{\beta}^{\brac{\frac{m}{2}-1}}(t).
\end{equation*}
Here $C_{\beta}^{\brac{\frac{m}{2}-1}}(t)$ is a Gegenbauer polynomial, which is defined as (see also \cite{AAR}, page 302)
\begin{equation*}
C_{\beta}^{\brac{\frac{m}{2}-1}}(t)=\sum_{j=0}^{b}(-1)^j\frac{\Gamma\brac{\beta-j+\frac{m}{2}-1}}{\Gamma\brac{\frac{m}{2}-1}j!(\beta-2j)!}(2t)^{\beta-2j}\ .
\end{equation*}
\end{theorem}
\begin{proof}
The expression $C^{\beta}\comm{1}$ is a polynomial depending on $\inner{u,x}$ which is harmonic in both $x$ and $u$. Polynomials that only depend on the inner product of two vectors $x$ and $u$ which belong to $\ker(\Delta_x,\Delta_u)$ are the so-called zonal harmonic polynomials wich are related to the Gegenbauer polynomial $\norm{x}^{2b}\norm{u}^{2b}C_{\beta}^{\brac{\frac{m}{2}-1}}(t)$. It is easily shown that 
\begin{equation*}
C^{\beta}\comm{1}=\inner{u,x}^{\beta} + l.o.t.\ ,
\end{equation*}
so we have to multiply $\norm{x}^{2b}\norm{u}^{2b}C_{\beta}^{\brac{\frac{m}{2}-1}}(t)$ with the constant $\frac{\beta !\, \Gamma\brac{\frac{m}{2}-1}}{2^{\beta}\Gamma\brac{\beta+\frac{m}{2}-1}}$ such that the term $t^{\beta}$ becomes monic. This completes the proof.
\end{proof}
\noindent
It follows from theorem \ref{Gegenbauer} that if the integer $\beta$ is odd, there is no constant term in $C^{\beta}\comm{1}$, which means that terms of the form $\norm{x}^{2a}\norm{u}^{2b}C^{2\beta+1}\comm{\mcH_0}$ will not contribute to the integration on the Stiefel manifold.
Therefore, we have that 
\begin{equation*}
\mcI\brac{P_{p,q}(x,u)}=\begin{dcases}
   \mcI\brac{H_0}  &\quad \text{if } p,q\in 2\mN \\
    0    &\quad \text{otherwise} \\
    \end{dcases}
\end{equation*}
To project $\mcP_{2p,2q}(\mR^{2m},\mC)\ker\big(\Delta_x,\Delta_u\big)$ onto $\mcH_0$, we need the following relation:
\begin{equation*}
A^{2\beta}C^{2\beta}\comm{1}=\prod_{i=1}^{2\beta}c_i(0,0)=(2\beta)!\frac{\brac{2\beta+\frac{m}{2}-1}}{\brac{\frac{m}{2}-1}}\brac{m-2}^{(2\beta)}. 
\end{equation*}
\noindent 
Putting together all the information, we finally arrive at the following result: 
\begin{theorem}
The integral of a polynomial $P(x,u)\in \mcP(\mR^{2m},\mC)$ on the Stiefel manifold is given by
\begin{equation*}
\int_{V_2(\mR^m)}P(\omega,\eta)dS_2=\sum_{i,j,k\in \mN}\gamma_{i}\left. \brac{A^{2i}\pi_{j,k} P(x,u)}\right |_{x=u=0}
\end{equation*}
 where $\pi_{j,k}$ is the projection operator from theorem \ref{Projection_operator_1}. The constant $\gamma_{i}$ is given by
\begin{equation*}
\gamma_{i}=\frac{(-1)^i}{4^{2i}i!\brac{\frac{m}{2}}^{(2i-1)}\brac{\frac{m}{2}+i-1}^{(i+1)}}
\end{equation*}
\end{theorem}
\begin{proof}
We will prove the statement for $P_{2p,2q}(x,u)\in\mcP_{2p,2q}\brac{\mR^{2m},\mC}$, since 
\begin{equation*}
\int_{V_2(\mR^m)}P_{p,q}(\omega,\eta)dS_2=0 \quad \text{if} \quad p,q\notin 2\mN.
\end{equation*}
Let us assume that $p\geq q$, as the proof for the other case is completely similar. As mentioned before, we first have to project onto the space $\mcP_{2(q-j),2(q-j)}\cap \ker\big(\Delta_x,\Delta_u\big)$ for $1\leq j \leq q$ using the projection operator $\pi_{p-q+j,j}$. From now on, we put $i:=q-j$. We then have to project on the trivial part $\mcH_0$, which can be done using $A^{2i}$. Since the trivial part is embedded in $\mcP_{2i,2i}(\mR^{2m},\mC)\cap \ker\big(\Delta_x,\Delta_u\big)$ as $C^{2i}\mcH_0$, we have 
\begin{align*}
A^{2i}C^{2i}\mcH_0&=\brac{\prod_{n=1}^{2i}c_n(0,0)}\mcH_0 \\
&=(2i)!\frac{\brac{2i+\frac{m}{2}-1}}{\brac{\frac{m}{2}-1}}\brac{m-2}^{(2i)}\mcH_0\ .
\end{align*}
This also means that 
\begin{equation*}
\frac{\brac{\frac{m}{2}-1}}{(2i)!\brac{2i+\frac{m}{2}-1}\brac{m-2}^{(2i)}}A^{2i}\mcP_{2i,2i}=\mcH_0\ .
\end{equation*}
Recall from theorem \ref{Gegenbauer} that $C^{2i}\comm{1}$ is, up to a constant multiple, a Gegenbauer polynomial in the variable $\inner{u,x}$. The Stiefel manifold consists of points $(x,u)\in \mR^{2m}$ with $\inner{u,x}=0$, so the only term in $C^{2i}\comm{1}$ that gives a non-trivial contribution to the integral is the constant term, which is given by 
\begin{equation*}
\frac{(2i)!\, \Gamma\brac{\frac{m}{2}-1}}{2^{2(q-j)}\Gamma\brac{2i+\frac{m}{2}-1}}C_{2i}^{(\frac{m}{2}-1)}(0)\ .
\end{equation*} 
All together we thus have the following contribution to the integral over the Stiefel manifold for $\mcP_{2i,2i}(\mR^{2m},\mC)\cap \ker\big(\Delta_x,\Delta_u\big)$:
\begin{equation*}
\int_{V_2(\mR^m)}P_{2i,2i}(\omega,\eta)dS_2=\frac{(-1)^i}{4^{2i}i!}\frac{A^{2i}P_{2i,2i}(x,u)}{\brac{\frac{m}{2}}^{(2i-1)}\brac{\frac{m}{2}+i-1}^{(i+1)}}
\end{equation*}
We then only have to sum up each of the parts to get the desired result. 
\end{proof}
\begin{remark}
In \cite{KK}, the authors also obtained a Pizzetti formula for integration on the Stiefel manifold. Altough their formula allows for easier and faster computation, our formula is more clear from a conceptual perspective. 
\end{remark}

\appendix
\section{Proof of theorem \ref{main-projection}}

In this appendix, we will complete the proof of theorem \ref{main-projection}:
\begin{theorem}
The projection operator 
\begin{equation*}
\Pi: \mcP_{k,l}(\mR^{2m},\mC) \cap \ker\big(\Delta_x,\Delta_u\big) \longrightarrow \mcH_{k,l}(\mR^{2m},\mC)\ , 
\end{equation*}
is defined as
\begin{equation*}
\Pi = \sum_{i,j=0}^{+\infty}\beta_{i,j}(k,l)C^iS_u^jA^iS_x^j,
\end{equation*}
where the coefficients $\beta_{i,j}(k,l)$ are given by: 
\begin{equation*}
\beta_{i,j}(k,l)=\frac{(-1)^{i+j}}{i!j!}\frac{(k+\frac{m}{2}-i+j-1)(l+\frac{m}{2}-j-3)_{(i)}}{(k+\frac{m}{2}+j-1)(l+\frac{m}{2}-3)_{(i)}(k+l+m-4)_{(i)}(k-l+2)^{(j)}}\ .
\end{equation*}
\end{theorem}
\noindent
Before we prove this theorem, we need some additional results concerning generalised hypergeometric functions.
\begin{definition}
The generalised hypergeometric function is defined as:
\begin{equation*}
\pFq{p}{q}{a_1,\ldots,a_p}{b_1,\ldots,b_q}{z}=\sum_{j=0}^{+\infty}\frac{(a_1)^{(j)}\ldots (a_p)^{(j)}}{(b_1)^{(j)}\ldots (b_q)^{(j)}}\frac{z^j}{j!},
\end{equation*}
where $a_i, b_j\in \mR$ for $1\leq i \leq p$ and $1\leq j\leq q$.
\end{definition}
\begin{definition}
The generalised hypergeometric function $\pFq{p}{q}{a_1,\ldots,a_p}{b_1,\ldots,b_q}{z}$ is called $k$-balanced $(k\in \mN)$ if and only if
\begin{equation*}
\sum_{i=1}^p a_i = \sum_{j=1}^q b_j -k.
\end{equation*}
\end{definition}
\begin{lemma}\label{hypergeom_lem1}
The following equality holds: 
\begin{align*}
\pFq{4}{3}{a,b,c,d}{e-1,f,g}{1}=&\ \pFq{4}{3}{a,b,c,d}{e,f,g}{1} \\
&+\frac{abcd}{(e-1)efg}\, \pFq{4}{3}{a+1,b+1,c+1,d+1}{e+1,f+1,g+1}{1}
\end{align*}
\end{lemma}
\begin{proof}
Follows from straightforward computations. 
\end{proof}
\noindent
A key lemma that will be used in the proof is known as Whipple's transformation. The proof of it can be found in e.g. \cite{Slater, Whipple}.
\begin{lemma}\label{Whipple}
If one of $z$ or $n$ is a positive integer and if the hypergeometric function is $1$-balanced, the following equality holds: 
\begin{align*}
 \pFq{4}{3}{a,b,-z,-n}{u,v,w}{1}=\ &\frac{\Gamma\brac{v+z+n}\Gamma\brac{w+z+n}\Gamma\brac{v}\Gamma\brac{w}}{\Gamma\brac{v+z}\Gamma\brac{v+n}\Gamma\brac{w+z}\Gamma\brac{w+n}} \\
&\times  \pFq{4}{3}{u-a,u-b,-z,-n}{u,1-v-z-n,1-w-z-n}{1}
\end{align*}
\end{lemma}
\noindent
As we want to apply Whipple's transformation to a $3$-balanced hypergeometric function, we need a transformation formula to transform a $3$-balanced hypergeometric function into a $1$-balanced. To obtain such a formula, we need the following lemma:
\begin{lemma}
The following equality holds for $n\in \mN$:
\begin{align*}
& \pFq{3}{2}{-1,a+c,b-c}{a+n,b+n}{z}\times \pFq{4}{3}{-n+1,a+b+n+1,a+c,b-c}{a+1,b+1,a+b+1}{z} \\
 &- \pFq{4}{3}{-n,a+b+n,a+c,b-c}{a+1,b+1,a+b+1}{z}=z(z-1)\frac{(1-n)(a+b+n+1)(a+c)(b-c)}{(a+1)(b+1)(a+n)(b+n)} \\
 &\times \pFq{4}{3}{-n+2,a+b+n+2,a+c+1,b-c+1}{a+2,b+2,a+b+1}{z}.
 \end{align*}
\end{lemma}
\begin{proof}
After some tedious, but straightforward computations, the left-hand side of the equation can be written as: 
\begin{align*}
\mathrm{LHS}=&\frac{1}{(a+n)(b+n)}\sum_{j=2}^{+\infty}C_j\frac{(1-n)_{(j-1)}(a+b+n+1)_{(j-1)}(a+c)_{(j-1)}(b-c)_{(j-1)}}{(a+1)_{(j-1)}(b+1)_{(j-1)}(a+b+1)_{(j-2)}(j-2)!}z^j \\
&+\frac{(1-n)(a+b+n+1)(a+c)(b-c)}{(a+1)(b+1)(a+n)(b+n)}z,
\end{align*}
where 
\begin{equation*}
C_j=\frac{(a+c)(b-c)}{(a+b+j-1)(j-1)}-\frac{(a+n)(b+n)(a+c+j-1)(b-c+j-1)}{(a+j)(b+j)(a+b+j-1)}. 
\end{equation*}
The right-hand side can be written as: 
\begin{align*}
\mathrm{RHS}=&\frac{1}{(a+n)(b+n)}\sum_{j=2}^{+\infty}D_j\frac{(1-n)_{(j-1)}(a+b+n+1)_{(j-1)}(a+c)_{(j-1)}(b-c)_{(j-1)}}{(a+1)_{(j-1)}(b+1)_{(j-1)}(a+b+1)_{(j-2)}(j-2)!}z^j \\
&+\frac{(1-n)(a+b+n+1)(a+c)(b-c)}{(a+1)(b+1)(a+n)(b+n)}z,
\end{align*}
where 
\begin{equation*}
D_j=\frac{(j-n)(a+b+n+j)(a+c+j-1)(b-c+j-1)}{(a+j)(b+j)(a+b+j-1)(j-1)}-1. 
\end{equation*}
A simple computation shows that $C_j=D_j$ for all integers $j$, which proves the lemma. 
\end{proof}
\noindent
Putting $z=1$ in the previous lemma, one obtains: 
\begin{corollary}\label{hypergeom_lem2}
The following equality holds for $n\in \mN$:
\begin{align*}
 \pFq{4}{3}{-n,a+b+n,a+c,b-c}{a+1,b+1,a+b+1}{1}=& \pFq{3}{2}{-1,a+c,b-c}{a+n,b+n}{1}\\ 
 &\times \pFq{4}{3}{-n+1,a+b+n+1,a+c,b-c}{a+1,b+1,a+b+1}{1}
 \end{align*}
Also note that: $ \pFq{3}{2}{-1,a+c,b-c}{a+n,b+n}{1}=1-\frac{(a+c)(b-c)}{(a+n)(b+n)}$.
\end{corollary}
\noindent
We will now prove the theorem.
\begin{proof}[Proof of theorem \ref{main-projection}]
It is clear that $\Pi\comm{\mcH_{k,l}}=\mcH_{k,l}$, so it remains to prove that  
\begin{equation*}
\Pi\comm{C^iS_u^j\mcH_{k-i+j,l-i-j}}=0 \iff G(k,l):= \sum_{a=0}^i\sum_{b=0}^j \alpha_{i,j}^{a,b}(k,l)\beta_{a,b}(k,l)=0,
\end{equation*}
where the coefficient $\alpha_{i,j}^{a,b}(k,l)$ are given in proposition 4.5. Using the following identity
\begin{equation*}
(x-j)^{(n)}=\frac{(x-1)_{(j)}(x)^{(n)}}{(x+n-1)_{(j)}},
\end{equation*}
the left-hand side can be written as a (double) sum that can be expressed as follows:
\begin{align*}
&G(k,l)=\sum_{b=0}^{j}  (-1)^b{j \choose b}\frac{(k-l+j+1)^{(b)}\brac{k+\frac{m}{2}-i-1}^{(b)}\brac{l+\frac{m}{2}-i-2}_{(b)}}{(k-l+2)^{(b)}\brac{k+\frac{m}{2}-1}^{(b)}\brac{l+\frac{m}{2}-2}_{(b)}} \\
&\times \pFq{6}{5}{-i, -K-L+i+3,-K-j+1,-K-b+2, -L+j+2, -L+b+3}{-K-L+4, -K+2, -K-b+1, -L+3, -L+b+2}{1},
\end{align*}
with $(K,L):=\brac{k+\frac{m}{2},l+\frac{m}{2}}$, as before. The hypergeometric function can be rewritten as a sum of hypergeometric functions of lower order. By doing so using some tedious but straightforward computations, we arrive at: 
\begin{equation}\label{TO DO}
G(k,l)=T_1+\frac{i(K+L-i-3)(K+j-1)(L-j-2)}{(K-1)(K-2)(L-2)(L-3)}\ T_2,
\end{equation}
where
\begin{align*}
T_1=&\ \pFq{4}{3}{-j,K-L+j+1,K-i-1,-L+i+2}{K-L+2,K-1,-L+2}{1}\\
&\times \pFq{4}{3}{-i, -K-L+i+3,-K-j+1,,-L+j+2}{-K-L+4, -K+2, -L+3}{1},
\end{align*}
and 
\begin{align*}
T_2=&\ \pFq{4}{3}{-j,K-L+j+1,K-i-1,-L+i+2}{K-L+2,K,-L+3}{1}\\
&\times\left(\pFq{4}{3}{-i+1, -K-L+i+4,-K-j+2,,-L+j+3}{-K-L+5, -K+3, -L+4}{1}\right. \\
&+\frac{(i-1)(K+L-i-4)(K+j-2)(L-j-3)}{(K+L-4)(K+L-5)(K-3)(L-4)}\\
&\left.\times \pFq{4}{3}{-i+2, -K-L+i+5-K-j+3,-L+j+4}{-K-L+6, -K+4, -L+5}{1}  \right).
\end{align*}
The term $T_2$ can be simplified using lemma \ref{hypergeom_lem1}. This leads to 
\begin{align*}
T_2=&\ \pFq{4}{3}{-j,K-L+j+1,K-i-1,-L+i+2}{K-L+2,K,-L+3}{1}\\
&\times\pFq{4}{3}{-i+1, -K-L+i+4,-K-j+2,,-L+j+3}{-K-L+4, -K+3, -L+4}{1}.
\end{align*}
Applying Whipple's transformation (lemma \ref{Whipple}) to the first factor of $T_1$ and to the second one of $T_2$, we get
\begin{align*}
G(k,l)=&\frac{\Gamma\brac{i+j}\Gamma\brac{-L+i+j+3}H(k,l)}{\Gamma\brac{-L+i+3}\Gamma\brac{K-L+j+2}} \pFq{4}{3}{-i,-j,-K-j+1,-K-i-1}{1-i-j,-L+2,L-i-j-2}{1},
\end{align*}
where $H(k,l)$ is given by
\begin{align*}
&H(k,l):=\frac{\Gamma\brac{K-L+2}\Gamma\brac{K-1}}{\Gamma\brac{i}\Gamma\brac{K+j-1}}\\ 
&\times \pFq{4}{3}{-i,-K-L+i+3,-K-j+1,-L+j+2}{-K-L+4,-K+2,-L+3}{1} \\
&+i(K+j-1)\frac{\Gamma\brac{K-L+i+j+1}\Gamma\brac{4-K-L}\Gamma\brac{L-K-i-j}\Gamma\brac{1-K}}{\Gamma\brac{j+1}\Gamma\brac{i+3-K-L}\Gamma\brac{L-K-j}\Gamma\brac{i+2-K}} \\
&\times \frac{\Gamma\brac{2-L}\Gamma\brac{L-1}}{\Gamma\brac{j+2-L}\Gamma\brac{L-i-1}} \ \pFq{4}{3}{-j,K-L+j+1,K-i-1,-L+i+2}{K-L+2,K,-L+3}{1}.
\end{align*}
Applying lemmas \ref{hypergeom_lem2} and \ref{Whipple} to $H(k,l)$, this reduces to:
\begin{align*}
&H(k,l):=(K-1)\frac{(-K+i+1)(-L+i+2)+(K+j-1)(-L+j+2)}{(-K+i+1)(-L+i+2)(-L+j+2)} \\
&\times \frac{\Gamma\brac{K-L+2}\Gamma\brac{K-1}\Gamma\brac{i+j}\Gamma\brac{-L+i+j+2}\Gamma\brac{-K+1}\Gamma\brac{-K-L+4}}{\Gamma\brac{i}\Gamma\brac{K+j-1}\Gamma\brac{-L+j+2}\Gamma\brac{-K-L+i+3}\Gamma\brac{j+1}\Gamma\brac{-K+i+1}}\\ 
&\times  \pFq{4}{3}{-i+1,-j+1,,K-i-1-K-j+1}{-L+3,1-i-j,L-i-j-1}{1} I(k,l)
\end{align*}
where $I(k,l)$ is given by
\begin{equation*}
I(k,l):=\frac{\Gamma\brac{K-L+i+j+1}\Gamma\brac{-K+L-i-j}\Gamma\brac{-L+2}\Gamma\brac{L-1}}{\Gamma\brac{K-L+j+1}\Gamma\brac{-K+L-j}\Gamma\brac{-L+i+2}\Gamma\brac{L-i-1}}-1
\end{equation*}
It is easy to show that $I(k,l)$ is zero using some basic gamma function properties, which completes the proof.  
\end{proof}



\end{document}